\documentclass[11pt,oneside,reqno]{amsart} 
\usepackage[width=15cm,height=24cm]{geometry}
\usepackage{amsfonts,amssymb,amsmath,amsthm}
\usepackage{graphicx}

\newtheorem{theorem}{Theorem}
\newtheorem{lemma}
           {Lemma}
\newtheorem{proposition}
           {Proposition}
           {Corollary}
           {Definition}
           {Exercise}
\newtheorem{example}
           {Example}
\newtheorem{remark}
           {Remark}

\newcommand{\e}{\mathbb{E}}

\newcommand{\IR}{{\mathbb R}}
\newcommand{\IN}{{\mathbb N}}
\newcommand{\IE}{{\mathbb E}}
\newcommand{\IO}{{\mathbb O}}

\newcommand{\norm}[1]{ \ensuremath { \| {#1} \| } }
\newcommand{\inner}[2]{ \ensuremath { \langle {#1},{#2} \rangle } }

\newenvironment{Proof of lemma}
               {\noindent{\bf Proof of Lemma}}{\hfill$\Box$\newline}
\newenvironment{Proof of theorem}{\noindent{\bf Proof of Theorem}}
               {\hfill\scriptsize{$\Box$}\newline}
\newenvironment{Proof of theorems}
               {\noindent{\bf Proof of Theorems}}{\hfill$\Box$\newline}
\newenvironment{Proof of proposition}
               {\noindent{\bf Proof of Proposition}}{\hfill$\Box$\newline}
\newenvironment{Proof of propositions}
               {\noindent{\bf Proof of Propositions}}{\hfill$\Box$\newline}
\newenvironment{Proof of exercise}
               {\noindent{\it Proof of Exercise:}}{\hfill$\Box$}
\newenvironment{Acknowledgements}
               {\noindent{\bf Acknowledgements.}}

\numberwithin{equation}{section}

\begin{document}

\title{Improved H\"older and reverse H\"{o}lder inequalities 
       for Gaussian random vectors.}

 \author[W.-K. Chen]{Wei-Kuo Chen}
\address{W.-K. Chen \\
         Department of Mathematics\\
         University of Chicago\\
         5734 S. University Avenue\\
         Chicago, IL 60637, USA}
\email{wkchen@math.uchicago.edu}

 \author[N. Dafnis]{Nikos Dafnis}
\address{N. Dafnis\\
         Department of Mathematics\\
         Texas A\&M University\\
         College Station, TX 77843, USA}
\email{nikdafnis@gmail.com} 
 
 \author[G. Paouris]{Grigoris Paouris}
\address{G. Paouris\\
         Department of Mathematics\\
         Texas A\&M University\\
         College Station, TX 77843, USA}
\email{grigorios.paouris@gmail.com}


\date{\today}
\begin{abstract}
We propose algebraic criteria that yield sharp H\"{o}lder types of inequalities for 
the product of functions of Gaussian random vectors with arbitrary covariance structure. 
While our lower inequality appears to be new, we prove that the upper inequality gives 
an equivalent formulation for the geometric Brascamp-Lieb inequality for Gaussian 
measures. As an application, we retrieve the Gaussian hypercontractivity as well as its reverse 
and we present a generalization of the sharp Young and reverse Young inequalities. 
From the latter, we recover several known inequalities in literatures including 
the Pr\'{e}kopa-Leindler and Barthe inequalities.
\end{abstract}
\maketitle

\section{Introduction and main results}

Let $(X_1,X_2)$ be a centered bivariate normal random vector and $f_1,f_2$ be any 
nonnegative measurable functions on $\mathbb{R}.$ What are the good upper and lower 
bounds for the expectation $\e f_1(X_1)f_2(X_2)$? Suppose that $p_1$ and $p_2$ are 
H\"{o}lder's conjugate exponents, 
\begin{align}
\label{eq0}
\frac{1}{p_1}+\frac{1}{p_2}=1.
\end{align}
The H\"{o}lder and reverse H\"{o}lder inequalities state that regardless the 
covariance between $X_1$ and $X_2,$ one always has
\begin{equation}
\label{eq-2}
\mathbb E f_{1}(X_{1})f_2(X_2)
\leq (\e f_1(X_1)^{p_1})^{\frac{1}{p_1}}(\e f_2(X_2)^{p_2})^{\frac{1}{p_2}}
\end{equation}
if $p_1,p_2\geq 1$ and 
\begin{equation}
\label{eq-1}
\mathbb E f_{1}(X_{1})f_2(X_2)
\geq (\e f_1(X_1)^{p_1})^{\frac{1}{p_1}}(\e f_2(X_2)^{p_2})^{\frac{1}{p_2}}
\end{equation}
if $0<p_1< 1$ and $p_2<0$. In this paper, we are interested in searching improved 
two-sided bounds that are related to the covariance of $(X_1,X_2)$ and can be 
easily used as H\"{o}lder's inequalities.

Our main result is stated as follows. Recall that a real symmetric $N\times N$ 
matrix $A$ is called positive definite (semi-definite) and denoted by $A>0$ 
$(\geq 0)$ if the usual inner product $\left<Ax,x\right>>0$ $(\geq 0)$ for all 
nonzero $x\in\mathbb{R}^N.$ For two real symmetric $N\times N$ matrices $A,B$, 
we say $B>A$ if $B-A>0$ and $B\geq A$ if $B-A\geq 0.$ 

\begin{theorem}\label{main}
Let $m,n_1,\ldots,n_m$ be positive integers and let $N=n_1+\cdots+n_m.$ Suppose 
that $X_i$ is a $n_i$-dimensional random vector for $1\leq i\leq m$ such that 
their joint law,
$$
{\bf X}:=(X_1,\ldots,X_m),
$$
forms a centered jointly $N$-dimensional Gaussian random vector with covariance 
matrix $T=(T_{ij})_{1\leq i,j\leq m}$, where $T_{ij}$ is the covariance matrix 
between $X_i$ and $X_j$ for $1\leq i,j\leq m$. Let $P$ be the block diagonal matrix,
$$
P={\rm diag}(p_1T_{11},\ldots,p_mT_{mm}).
$$
For any set of nonnegative measurable functions $f_i$ on $\mathbb{R}^{n_i}$ for 
$1\leq i\leq m$, the following statements hold.
 \begin{itemize}
  \item[$(i)$]  If $T\leq P$, then
              \begin{equation}\label{main1}
                \IE \prod_{i=1}^m f_i(X_i)\leq
                \prod_{i=1}^m \Big(\IE f_i(X_i)^{p_i}\Big)^{\frac{1}{p_i}}.
              \end{equation}
  \item[$(ii)$] If $T\geq P$, then
              \begin{equation}\label{main2}
                \IE \prod_{i=1}^m f_i(X_i)\geq
                \prod_{i=1}^m \Big(\IE f_i(X_i)^{p_i}\Big)^{\frac{1}{p_i}}.
              \end{equation}
 \end{itemize}
\end{theorem}
Here the right-hand sides of \eqref{main1} and \eqref{main2} adapt the convention that $\infty\cdot 0=0$ whenever such situation occurs, which will remain in force throughout the rest of the paper.

\begin{remark}\label{rmrk.Multi-Chen1}\rm
 Suppose that $0<\e f_i(X_i)^{p_i}<\infty$ for $1\leq i\leq m$ and at least one of 
 $f_i$'s is not equal to a constant almost everywhere. Then we get strict inequalities 
 in \eqref{main1} if $T<P$ and in \eqref{main2} if $T>P$. To see this, take $T>P$ for 
 instance. This allows us to find $q_1,\ldots,q_m$ with $q_1>p_1,\ldots,q_m>p_m$ such 
 that $Q:={\rm diag}(q_1T_{11},\ldots,q_mT_{mm})$ satisfies $T>Q>P$. From Jensen's 
 inequality, $(\e f_i(X_i)^{p_i})^{1/p_i}\leq (\e f_i(X_i)^{q_i})^{1/q_i}$ and this 
 inequality is strict if $f_i$ is not a.s. a constant. So \eqref{main2} yields 
 $$
 \prod_{i=1}^m\IE f_i(X_i) \geq 
 \prod_{i=1}^m(\e f_i(X_i)^{q_i})^{1/q_i}>\prod_{i=1}^m(\e f_i(X_i)^{p_i})^{1/p_i}.
 $$
\end{remark}

\begin{remark}\label{rmrk2}\rm
 If inequality \eqref{main1} {\rm (}resp. \eqref{main2}{\rm )}
 holds for all nonnegative $f_1,\ldots,f_m$, then we get that $T\leq P$
 {\rm (}resp. $T\geq P${\rm )}. This can be seen by using the test functions
 $f_i(x_i)=e^{\inner{\alpha_i}{x_i}}$ for $\alpha_i\in\mathbb{R}^{n_i}.$ 
 A direct computation gives that for ${\alpha}=(\alpha_1,\ldots,\alpha_m),$
 \begin{align}\label{rmrkMulti:eq1}
   \IE \prod_{i=1}^mf_i(X_i)&=\exp \frac{1}{2}
   \inner{T\alpha}{\alpha} \quad {\rm and} \quad
   \prod_{i=1}^m (\e f_i(X_i)^{p_i})^{\frac{1}{p_i}}
  =\exp\frac{1}{2}\inner{P\alpha}{\alpha} .
 \end{align}
  Thus, if (\ref{main1}) holds for all nonnegative functions, then we get
 that $\left<T{\alpha},{\alpha}\right>\leq \left<P{\alpha},{\alpha}\right>$ for any
 ${\alpha}\in\mathbb{R}^N$ and so $T\leq P$. Similarly, if (\ref{main2}) holds
 true for all nonnegative functions, then
 $T\geq P$.
\end{remark}

\begin{remark}\label{rmk3}\rm
If ${\alpha}=(\alpha_1,\ldots,\alpha_m)\in\mbox{Ker}(T-P)$ for 
$\alpha_i\in\IR^{n_i}$, then equalities hold in \eqref{main1} 
and \eqref{main2} when $f_i(x_i)=e^{\inner{\alpha_i}{x_i}}$ for $1\leq i\leq m.$
\end{remark}

\smallskip 

Let us now illustrate how our theorem recovers H\"{o}lder and reverse H\"{o}lder inequalities. Let $(X_1,X_2)$ be a centered non-degenerate bivariate normal random vector with covariance matrix $T=(T_{ij})_{1\le i,j\leq 2}$. Suppose 
that $p_1,p_2\neq 1$ satisfy H\"{o}lder's condition \eqref{eq0}. Set $P={\rm diag}(p_1T_{11},p_2T_{22})$.
Note that \eqref{eq0} implies 
\begin{equation}\label{eq.p1p2}
 (p_1-1)(p_2-1)= 1 - (\,p_1+p_2-p_1p_2\,) = 1.
\end{equation}
Thus, for $p_1,p_2 >1$, we have that 
\begin{align*}
  P-T&=\left(
  \begin{array}{cc}
   T_{11}(p_1-1)&T_{12}\\
   T_{12}&T_{22}(p_2-1)
  \end{array}
  \right) \geq 0 
\end{align*}
since using \eqref{eq.p1p2} gives $\det(P-T)=\det(T)\geq 0 $.
This shows that Theorem \ref{main}(i) implies H\"{o}lder inequality. Similarly, if $p_1,p_2 <1$, 
\eqref{eq.p1p2} yields $T-P \geq 0$ and then Theorem \ref{main}(ii) implies the reverse H\"{o}lder inequality. 

To see why Theorem \ref{main} improves H\"{o}lder's bounds in general, assume $\det(T)>0$ and again $p_1,p_2\neq 1$ satisfy \eqref{eq0}. Let $f_1$ and $f_2$ be any two nonnegative measurable functions such that at least one of them is not equal to a
constant a.e. and $0<(\e f_1(X_1)^{p_1})^{1/p_1}(\e f_2(X_2)^{p_2})^{1/p_2}<\infty$. First we consider the case $p_1,p_2>1.$ Observe that for any $q_1\in [1,p_1)$ and $q_2\in [1,p_2)$, we have $Q-T\geq 0$ if and only if $\det(Q-T)\geq 0$, where $Q:={\rm diag}(q_1T_{11},q_2T_{22})$. Write
\begin{align*}
\det(Q-T)&= \det(T) - \varepsilon_Q T_{11}T_{22},
\end{align*}
where $\varepsilon_Q:= q_1+q_2-q_1q_2.$ Note that $\varepsilon_Q\rightarrow 0$ when $q_1\uparrow p_1$ and $q_2\uparrow p_2$. Since $\det(T) >0$, there exist
exponents $q_1\in[1,p_1)$ and $q_2\in[1,p_2)$ such that $T \leq Q < P$, which implies from Theorem 
\ref{main}(i) and then Jensen's inequality that
\begin{align*}
\e f_1(X_1)f_2(X_2)& \,\leq \,
\big(\e f_1(X_1)^{q_1}\big)^{\frac{1}{q_1}}\,
\big(\e f_2(X_1)^{q_2}\big)^{\frac{1}{q_2}}\,<\,
\big(\e f_1(X_1)^{p_1}\big)^{\frac{1}{p_1}}\,
\big(\e f_2(X_1)^{p_2}\big)^{\frac{1}{p_2}}.
\end{align*}
Similarly, if $p_1,p_2<1$, there exist $q_1\in(p_1,1]$ 
and $q_2\in(p_2,1]$ such that $T \geq Q > P$, which implies from Theorem \ref{main}(ii) and again Jensen's inequality that
\begin{align*}
\e f_1(X_1)f_2(X_2)& \,\geq \,
\big(\e f_1(X_1)^{q_1}\big)^{\frac{1}{q_1}}\,
\big(\e f_2(X_1)^{q_2}\big)^{\frac{1}{q_2}}\,>\,
\big(\e f_1(X_1)^{p_1}\big)^{\frac{1}{p_1}}\,
\big(\e f_2(X_1)^{p_2}\big)^{\frac{1}{p_2}}.
\end{align*}
In other words, the exponents $q_1,q_2$ in each case improve H\"{o}lder's bounds.

\smallskip 

\begin{example}\rm
Assume that $m=2$ and $X_1,X_2$ are standard Gaussian with $\e X_1X_2=t$ for $0\leq t\leq 1$. The simplest H\"{o}lder's types of bounds for $\e f_1(X_1)f_2(X_2)$ can be obtained as follows. Note that $ (1-t) I_{2} \leq T \leq (1+t)I_{2}.$ 
Theorem \ref{main} gives that for $q_t:=1-t$ and $p_t:=1+t,$
\begin{equation}\label{eq.PrEx}
\big(\mathbb E f_{1}(X_{1})^{q_t}\big)^{\frac{1}{q_t}}\,
\big(\mathbb E f_{2}(X_{2})^{q_t}\big)^{\frac{1}{q_t}}
\,\leq\, \mathbb E f_{1}(X_{1}) f_{2}(X_{2}) \,\leq\,
\big(\mathbb E f_{1}(X_{1})^{p_t}\big)^{\frac{1}{p_t}}\,
\big(\mathbb E f_{2}(X_{2})^{p_t}\big)^{\frac{1}{p_t}}
\end{equation} for any nonnegative measurable functions $f_1,f_2.$ In particular, if $t=0,$ then $X_1,X_2$ are independent and the three quantities in \eqref{eq.PrEx} are the same; if $t=1$, the left-hand side is the Jensen inequality and the right-hand side gives the Cauchy-Schwartz inequality. To see the sharpness of \eqref{eq.PrEx}, note that $(1,1)\in\mbox{Ker}(T-(1+t)I_2)$ and $(1,-1)\in\mbox{Ker}(T-(1-t)I_2).$ From Remark \ref{rmk3}, 
$f_{1}(x)=f_{2}(x)= e^{x}$ give the left-hand sided equality of \eqref{eq.PrEx}, while the functions $f_{1}(x)=e^{x}$ and $f_{2}(x)=e^{-x}$ give the equality for the other side.
\end{example}

Inequality \eqref{main1} is strongly related to the famous Brascamp-Lieb inequality,
firstly proved by Brascamp and Lieb in \cite{BL} and later fully generalized
by Lieb in \cite{Lieb1}. It says that if $m\geq n$, $p_{1}, \ldots, p_{m}\geq 1$ with 
$\sum_{i=1}^{m}{n_{i}}p_i^{-1} = n$ and ${U}_i$ 
is a surjective linear map from $\IR^n$ to $\IR^{n_i}$ for $1\leq i\leq m$, then for 
any set of nonnegative $f_{i} \in L_{p_{i}}(\mathbb R^{n_i})$ for $1\leq i\leq m$, the 
ratio
\begin{align}\label{BL} 
\frac{ \int_{ \mathbb R^{n}}
\prod_{i=1}^{m} f_{i}({U}_ix) d x }{ \prod_{i=1}^{m} \|f_{i}\|_{p_{i}} }
\end{align}
is maximized by centered Gaussian functions, i.e., functions of the form
$$f_{i}(x_i)= \exp(-\inner{A_ix_i}{x_i}),$$ where $A_i$
is a $n_i\times n_i$-dimensional real symmetric and positive definite matrix. 
Their approach was based on a tensorization argument and the Brascamp- Lieb-Luttinger  
rearrangement inequality \cite{BLL}. Several different proofs of this inequality 
have appeared later using different tools, see \cite{B,BC,BBC,BCCT,Le}. For more 
information about the Brascamp-Lieb inequalities as well as their generalizations in 
non-Euclidean settings, we refer to \cite{BCLM,BH,C,CC,CLL,Lieb1}. 

Among various formulations, Ball first put forward the geometric form of the 
Brascamp-Lieb inequality in \cite{Ball1} and used it to derive sharp inequalities for convex bodies in $\mathbb{R}^n$ (see \cite{Ball2}). Later it was generalized by Barthe \cite{B} and in the recent paper \cite{BCLM}, Bennett, Carbery, Christ and Tao showed that by 
a clever change of variables, one can retrieve the initial Brascamp-Lieb inequality 
by this geometric form. 

For the purpose of our discussion, we shall also consider an equivalent version of 
the geometric Brascamp-Lieb inequality (Theorem \ref{G-BL} in Section \ref{sec3}), where 
the underlying measures are Gaussian. We will show that Theorem \ref{main}(i) is 
indeed another formulation of this inequality. While the connection between the upper 
bound \eqref{main1} and the Brascamp-Lieb inequality can be completely clarified, the 
lower bound \eqref{main2} appears to be new to the authors as it is by no means clear 
which known equalities will imply \eqref{main2}.

\smallskip

In this paper, we will present two applications from Theorem \ref{main} with proper chosen 
covariance matrices and exponents. The first is Nelson's Gaussian hypercontractivity and its reverse 
form. The second is the Lebesgue version of Theorem \ref{main} that provides a generalization of the sharp Young and reverse Young inequalities, for which we now formulate.

\begin{theorem}\label{Gen-Rev-Young}
Let $n,m\in\IN$, $n_1,\ldots,n_m \leq n$ and $p_{1}, \ldots, p_{m}$ be real numbers such that
\begin{equation}\label{eq.homogen}
\sum_{i=1}^{m} \frac{n_{i}}{p_i} = n,
\end{equation}
Assume that  $U_{i}$ is a $n_{i}\times n$ matrix with rank $n_{i}$ for $1\leq i\leq m$. 
Set $N=\sum_{i=1}^m n_i$. Let $U$ be the $N\times n$ matrix with block rows 
$U_1,\ldots,U_m,$ i.e., $U^*=\left(U_1^*,\ldots,U_m^*\right)$. Let $B$ be a $n\times n$ 
real symmetric and positive definite matrix. Set
\begin{align*}
P&={\rm diag}\left(p_1I_{n_1},\ldots,p_mI_{n_m}\right),\\
D_{UBU^*}&={\rm diag}\left(U_1B U_1^*,\ldots,U_mBU_m^*\right).
\end{align*}
For nonnegative $f_{i} \in L_{p_{i}}(\IR^{n_i})$ for $i\leq m$, 
the following statements hold.
\begin{itemize}
 \item[(i)] If
            \begin{equation}\label{eq.A-R-Y-Up}
             UBU^* \leq P D_{UBU^*},
            \end{equation}
            then
            \begin{equation} \label{eq.R-Y-Up-intro}
             \int_{\mathbb R^{n}} \prod_{i=1}^{m} f_{i} (U_i\,x)\,dx
             \leq  \left(\frac{{\rm det}(B)}
             {\prod_{i=1}^m{\rm det}(U_{i}B{U}_i^*)^{\frac{1}{p_i}}}\right)^{\frac12}
             \,\prod_{i=1}^m \|f_{i} \|_{p_{i}}.
            \end{equation}
            The equality holds if $f_i(x_i)=
            \exp\left(-p_i^{-1}\left<(U_iBU_i^*)^{-1}x_i,x_i\right>\right)$ for $i\leq m$.
 \item[(ii)] If
            \begin{equation}\label{eq.A-R-Y-Down}
             UBU^* \geq P D_{UBU^*},
            \end{equation}
then
            \begin{equation} \label{eq.R-Y-Down-intro}
             \int_{\mathbb R^{n}} \prod_{i=1}^{m} f_{i} (U_i\,x)\,dx
             \geq  \left(\frac{{\rm det}(B)}
             {\prod_{i=1}^m{\rm det}(U_{i}BU_i^*)^{\frac{1}{p_i}}}\right)^{\frac12}
             \,\prod_{i=1}^m \|f_{i} \|_{p_{i}}.
            \end{equation}
\end{itemize}
\end{theorem}

As we will see in Sections \ref{sec6} and \ref{Sec6}, \eqref{eq.R-Y-Up-intro} is indeed the 
Brascamp-Lieb inequality, only now the original geometric condition in Ball's geometric 
Brascamp-Lieb inequality is equivalently replaced by the algebraic inequality 
\eqref{eq.A-R-Y-Up}. However, once again the authors do not know  which known inequalities 
in the literatures are equivalent to \eqref{eq.R-Y-Down-intro}.  Neither do they know the 
form of the $f_i$'s that yields the equality in \eqref{eq.R-Y-Down-intro} in general. 

\smallskip

There are several consequences that can be drawn from Theorem \ref{Gen-Rev-Young}. They 
all start from a generalization of Barthe's lemma. In \cite{BAYoung}, Barthe used 
measure transportation techniques to give a simple proof of the sharp Young and reverse 
Young inequalities. Later in \cite{BaArxiv}, he generalized the argument and derived a 
reverse form of the Brascamp-Lieb inequality \eqref{BL}, known as Barthe's inequality. 
The core of \cite{BAYoung,BaArxiv} was played by a moment inequality stated in Lemma 1 \cite{BAYoung}, which we call Barthe's lemma. 
Using Theorem \ref{Gen-Rev-Young}, we will establish a generalization of his lemma (see Theorem \ref{THE prop} in Section \ref{sec6}) that yields a two-sided moment inequality. 
Incidentally, similar results are discovered independently in a recent work of Barthe and 
Wolff \cite{BW} using again measure transportation methods. 

\smallskip

The power of our result could be borne upon the fact that it indeed implies several 
inequalities as:
\begin{itemize}
\item[(I)] The Pr\'{e}kopa-Leindler inequality 
\item[(II)] The sharp Young and reverse Young inequalities.
\item[(III)] The Brascamp-Lieb and Barthe inequalities.
\item[(IV)] An entropy inequality (see Subsection 6.3 below).
\end{itemize}
For a comprehensive overview about the connections among these and other known 
inequalities in literatures, we refer to the survey paper of Gardner \cite{Gardner}.

\smallskip

The rest of the paper is organized as follows. 
 Section \ref{Sec2} is devoted to proving Theorem \ref{main}, where two different
proofs are presented. The first is based on the Gaussian integration by parts formula 
combining with an iteration argument and the second uses the Ornstein-Uhlenbeck semi-group 
techniques. 
 In Section \ref{sec4}, we investigate the connection between Theorem \ref{main}(i) and the 
Brascamp-Lieb inequality, we study the geometry of the eligible exponents in Theorem 
\ref{main}(i) and we prove Nelson's hypercontractivity  and its reverse.
 In Section \ref{sec5}, we prove Theorem \ref{Gen-Rev-Young} and explain how it generalizes the sharp Young and reverse Young inequalities.
 In Section \ref{sec6}, we prove the generalized Barthe's lemma. 
Finally, its applications are given in Section \ref{Sec6}, where we deduce 
inequalities (I), (II), (III) and (IV).

\section{Proofs of Theorem \ref{main}}\label{Sec2}

In this section, we will present two fundamental proofs for Theorem \ref{main}. 
Let $m,n_1,\ldots,n_m$ be positive integers and let $N=n_1+\cdots+n_m.$ Recall 
$\mathbf{X},X_1,\ldots,X_m$ and the matrices $T,P$ from the statement of Theorem 
\ref{main}. We denote $X_i=(X_{i1},\ldots,X_{in_i}).$ An application of change of 
variables suggests that in both proofs, we may assume without loss of generality, 
$T_{11}= I_{n_{1}},\ldots,T_{mm}=I_{n_m}.$ In other words, each $X_i$ is a 
$n_i$-dimensional standard Gaussian random vector. For notational convenience, 
we use $\gamma_k$ to denote the $k$-dimensional standard Gaussian measure on 
$\mathbb{R}^k$ for $k\geq 1.$

\subsection{First proof: the Gaussian integration by parts}
We begin with the formulation of the Gaussian integration by parts formula. Let 
$Y,Z_1,\ldots,Z_N$ be centered jointly Gaussian random variables.
For a real-valued function $F$ defined on $\mathbb{R}^N$ with uniformly bounded 
first partial derivatives, this formula reads
\begin{align*}
\IE YF(Z_1,\ldots,Z_N)&=\sum_{i=1}^N\IE YZ_i\cdot
\IE\frac{\partial F}{\partial x_i}(Z_1,\ldots,Z_N).
\end{align*}
 This formula has been playing
a fundamental role in understanding the behavior of the highly correlated Gaussian
random variables arising from modeling various scientific phenomenon, such as the mean
field spin glass models \cite{Tal11}. The argument that we are about to present below 
has already been applied to quantify the error estimate in similar inequalities. We 
refer to \cite{Chen11} along this direction. Set $A=T-P.$ We now state a lemma that 
is the real ingredient of the matter.

\begin{lemma}\label{lem1} Let $L_1,\ldots,L_m$ be real-valued functions defined
respectively on $\mathbb{R}^{n_1},\ldots,\mathbb{R}^{n_m}$
and their first four partial derivatives be uniformly bounded. Define for $u\in [0,1],$
\begin{align*}
\phi(u)&=\log \IE\exp \sum_{i=1}^mL_i(\sqrt{u}X_i)
\end{align*}
and
\begin{align*}
\phi_i(u)&=
\frac{1}{p_i}\log \IE\exp p_i L_i(\sqrt{u}X_i)
\end{align*}
for $1\leq i\leq m,$ where $\phi_i$ should be read as $\e L_i(\sqrt{u}X_i)$ when $p_i=0.$ If $A\leq 0,$ then
\begin{align}
\label{lem1:eq1}
\phi(u)&\leq \sum_{i=1}^m\phi_i(u)+Ku^2;
\end{align}
if $A\geq 0,$ then
\begin{align}
\label{lem1:eq2}
\phi(u)&\geq \sum_{i=1}^m\phi_i(u)-Ku^2.
\end{align}
Here $K$ is some positive constant depending only on the supremum norms of the first 
four partial derivatives of $L_1,\ldots, L_m.$
\end{lemma}

\begin{proof} 
For clarity, we adapt the notation $x_i=(x_{i1},\ldots,x_{in_i})\in\mathbb{R}^{n_i}$ 
and $\partial_{x_{ij}}L_i(x_i)$ standards for the partial derivative of $U_i$ with 
respect to $x_{ij}.$ Using the relation $\e X_{ij}X_{ij'}=\delta_{j,j'}$ for all 
$1\leq i\leq m$ and every $1\leq j,j'\leq n_i,$ a direct computation using Gaussian 
integration by parts and $T_{ii}=I_{n_i}$ yields
\begin{align*}
&\phi'(u)\\
&=\frac{1}{2\sqrt{u}}\frac{\IE
\sum_{i=1}^m\sum_{j=1}^{n_i}X_{ij}
\partial_{x_{ij}}L_i(\sqrt{u}X_i)\exp \sum_{k=1}^mL_k(\sqrt{u}X_k)}{\exp\phi(u)}\\
&=\frac{1}{2\exp \phi(u)}\biggl(\sum_{i=1}^m\sum_{j=1}^{n_i}
\IE \partial_{x_{ij}}^2L_i(\sqrt{u}X_i)\exp \sum_{k=1}^mL_k(\sqrt{u}X_k)\,+\bigg.\\
&\biggl.\sum_{i,i'=1}^m\sum_{j=1}^{n_i}\sum_{j'=1}^{n_{i'}}\IE X_{ij}X_{i'j'}
\IE \partial_{x_{ij}}L_i(\sqrt{u}X_i)
\partial_{x_{i'j'}}L_{i'}(\sqrt{u}X_i')\exp \sum_{k=1}^mL_k(\sqrt{u}X_k)\biggr)
\end{align*}
and
\begin{align*}
\phi_i'(u)&=\frac{1}{2\sqrt{u}}
\frac{\IE \sum_{j=1}^{n_i}X_{ij}
\partial_{x_{ij}}L_i(\sqrt{u}X_i)\exp p_i L_i(\sqrt{u}X_i)}{\exp p_i \phi_i(u)}\\
&=\frac{1}{2}\sum_{j=1}^{n_i}\frac{\IE
(\partial_{x_{ij}}^2L_i(\sqrt{u}X_i)+p_i
(\partial_{x_{ij}}L_i(\sqrt{u}X_i))^2)\exp p_i L_i(\sqrt{u}X_i)}{\exp p_i\phi_i(u)}.
\end{align*}
Thus,
\begin{align}
\begin{split}\nonumber
\phi'(0)-\sum_{i=1}^n\phi_i'(0)
&=\frac{1}{2}\sum_{i,i'=1}^m\sum_{j=1}^{n_i}
\sum_{j'=1}^{n_{i'}}\IE X_{ij}X_{i'j'}\partial_{x_{ij}}L_{i}(0)\partial_{x_{i'j'}}L_{i'}(0)\\
&-\frac{1}{2}\sum_{i=1}^m\sum_{j=1}^{n_i}p_i(\partial_{x_{ij}}L_i(0))^2
\end{split}\\
\begin{split}\label{eq1}
&=\frac{1}{2}\left<AV,V\right>,
\end{split}
\end{align}
where
$$
V=(\partial_{x_{11}}L_1(0),\ldots,
\partial_{x_{1n_1}}L_1(0),\ldots,\partial_{x_{m1}}L_m(0),\ldots,\partial_{x_{mn_m}}L_m(0)).
$$
One may also perform a similar computation as above to represent the second derivatives
of $\phi_1,\ldots,\phi_m,\phi$ in terms of the first four partial derivatives of 
$L_1,\ldots,L_m$ using Gaussian integration by parts formula. From the uniformly 
boundedness of the first four partial derivatives of $L_1,\ldots, L_m$, we have that
$$
\sup_{0\leq u\leq 1}\left|\phi''(u)-\sum_{i=1}^m\phi_i''(0)\right|\leq K
$$
for some fixed positive constant $K.$ Using this, $(\ref{eq1})$ and
$\phi(0)=\sum_{i=1}^m\phi_i(0),$ we conclude from the mean value theorem that 
if $A\leq 0,$ then \eqref{lem1:eq1} follows since
\begin{align*}
\phi(u)-\sum_{i=1}^m\phi_i(u)&\leq\phi(0)-
\sum_{i=1}^m\phi_i(0)+\left(\phi'(0)-\sum_{i=1}^m\phi_i'(0)\right)u+Ku^2\\
&=0+\frac{1}{2}\left<AV,V\right>u+ Ku^2\\
&\leq Ku^2.
\end{align*}
Similarly, if $A\geq 0,$ we also obtain \eqref{lem1:eq2} and this completes our proof.
\end{proof}

\begin{proof}[Proof of Theorem \ref{main}:] To avoid triviality, we will assume that each $f_i$ is not identically zero and each $T_{ii}$ is not a zero matrix. Our arguments will be divided into two major parts. First, we consider the case that
$f_1=\exp{L_1},\ldots, f_m=\exp{L_m}$, where $L_1,\ldots,L_m$ are defined respectively on
$\mathbb{R}^{n_1},\ldots,\mathbb{R}^{n_m}$ with uniformly bounded partial derivatives of 
any orders. Define
\begin{align*}
\phi(u,x_1,\ldots,x_m)&=\log\mathbb{E}\exp
\sum_{i=1}^mL_i(x_i+\sqrt{u}X_i)
\end{align*}
 and
\begin{align*}
\phi_i(u,x_i )&=
\frac{1}{p_i}\log\mathbb{E}\exp p_iL_i(x_i +\sqrt{u}X_i),
\end{align*}
for $u\in[0,1]$ and $x_i=(x_{i1},\ldots,x_{in_i})\in\mathbb{R}^{n_i}$, where $\phi_i$ is read as $\e L_i(x_i+\sqrt{u}X_i)$ when $p_i=0.$ 
We prove \eqref{main1} first. Note that since 
the first four partial derivatives of $L_1,\ldots,L_m$ are uniformly bounded, one can  
use the Gaussian integration by parts formula as we have done in Lemma $\ref{lem1}$ to  
obtain a constant $K>0$ independent of $u,x_1,\ldots,x_m$ such that the first four 
partial derivatives of $\phi_1(u,x_1+\cdot),\ldots,\phi_m(u,x_m+\cdot)$ are uniformly 
bounded by $K$. Let $K'$ be the constant obtained by applying $(\ref{lem1:eq1})$ to
$\phi_1(u,x_1+\cdot),\ldots,\phi_m(u,x_m+\cdot)$ instead of $L_1,\ldots,L_m,$
i.e. $K'$ satisfies that
\begin{align}
\begin{split}\label{eq4}
&\log\IE\exp \sum_{i=1}^m\phi_i(u,x_i+\sqrt{v}X_i)\\
&\leq
\sum_{i=1}^m\frac{1}{p_i}\log\IE\exp p_i\phi_i(u,x_i+\sqrt{v}X_i)+K'v^2.
\end{split}
\end{align}
Note that $K'$ only depends on $K.$

\smallskip

We claim that for every $M\in\mathbb{N},$
\begin{align}\label{thm:proof:eq1}
\phi\left(\frac{j}{M},x_1,\ldots,x_m\right)
\leq \sum_{i=1}^m\phi_i\left(\frac{j}{M},x_i\right)+\frac{K'j}{M^2}
\end{align}
for all $1\leq j\leq M.$
Since $\phi_i(0,\cdot)=L_i(\cdot)$ for $1\leq i\leq m,$ the base case $j=1$ follows 
by letting $u=0$ and $v=1/M$ in $(\ref{eq4}).$ Suppose that our claim holds for some 
$1\leq j\leq M-1$. Write
\begin{align*}
\phi\left(\frac{j+1}{M},x_1,\ldots,x_m\right)&=\log\mathbb{E}
\exp\phi\left(\frac{j}{M},
x_1+\frac{X_1}{\sqrt{M}},\ldots,x_m+\frac{X_m}{\sqrt{M}}\right).
\end{align*}
Using the induction hypothesis and then $(\ref{eq4})$ with $u=j/M$ and $v=1/M$, we have
\begin{align*}
&\phi\left(\frac{j+1}{M},x_1,\ldots,x_m\right)\\
&\leq \log\mathbb{E}\exp
\sum_{i=1}^m\phi_i\left(\frac{j}{M},x_i+\frac{X_i}{\sqrt{M}}\right)+\frac{K'j}{M^2}\\
&\leq \sum_{i=1}^m\frac{1}{p_i}
\log\mathbb{E}\exp p_i\phi_i\left(\frac{j}{M},x_i+\frac{X_i}{\sqrt{M}}\right)
+\frac{K'}{M^2}+\frac{K'j}{M^2}\\
&=\sum_{i=1}^m\phi_i\left(\frac{j+1}{M},x_i\right)+\frac{K'(j+1)}{M^2}.
\end{align*}
This completes the proof of our claim. Now, letting $j=M$ and $x_1,\ldots,x_m$
be all equal to the zero vectors in $(\ref{thm:proof:eq1})$ and $M\rightarrow\infty$, 
we obtain \eqref{main1} in the case that $f_1=\exp L_1,\ldots,f_m=\exp L_m$. One may argue similarly to obtain \eqref{main2} in such case as well.

\smallskip

Next we consider the general case that $f_1,\ldots,f_m$ are nonnegative measurable. Note that in the following, $(\e f(Y)^{p})^{1/p}$ will always be read as $\exp\e \log f(Y)$ whenever $p=0$ and the latter is well-defined. First, we assume that for every $1\leq i\leq m,$ 
\begin{align}\label{eq88}
\mbox{$\e f_i(X_i)^{p_i}<\infty$ if $p_i\neq 0$ and $\e \log f_i(X_i)>-\infty$ if $p_i=0$.}
\end{align} 
Under this assumption, from the monotone convergence theorem, it suffices to assume that $1/2\leq f_1,\ldots,f_m\leq 1$. Let
$$
f_{i,j}=\frac{1}{2}1_{([-j,j]^{n_i})^c}+f_j1_{[-j,j]^{n_i}}.
$$
Since $f_{i,j}\uparrow f_i$ as $j\rightarrow \infty$, we can further assume by the monotone convergence
theorem that $f_i=1/2$ on $([-1,1]^{n_i})^c$ and $1/2\leq f_i\leq 1$ on $[-1,1]^{n_i}.$ Now we
use mollifier function to construct a sequence of smooth functions $(g_{i,j})_{j\geq 1}$ that
satisfies $g_{i,j}=1/2$ on $([-3/2,3/2]^{n_i})^c$, $1/2\leq g_{i,j}\leq 1$ on $[-3/2,3/2]^{n_i}$,
and converges to $f_i$ a.e. with respect to the Lebesgue measure. Therefore, with these constructions,
\begin{align}
\begin{split}\label{thm1:proof:eq1}
\lim_{j\rightarrow\infty}\bigl(\IE g_{i,j}(X_i)^{p_i}\biggr)^{\frac{1}{p_i}}&=\biggl(\IE f_i(X_i)^{p_i}\bigr)^{\frac{1}{p_i}},\\
\lim_{j\rightarrow\infty}\e\prod_{i=1}^mg_{i,j}(X_i)&=\e\prod_{ i=1}^mf_i(X_i).
\end{split}
\end{align}
Take $L_{i,j}=\log g_{i,j}.$ Then each $L_{i,j}$ has uniformly bounded derivatives of any orders
and $g_{i,j}=\exp{{L_{i,j}}}.$ By the first part of our argument and $(\ref{thm1:proof:eq1})$, we obtain \eqref{main1} and \eqref{main2}.

\smallskip

To finish the proof, it remains to deal with the case that \eqref{eq88} does not hold for all $1\leq i\leq m.$ Let 
\begin{align*}
I&=\{i:p_i>0,\e f_i(X_i)^{p_i}=\infty\},\\
I'&=\{i:p_i>0,\e f_i(X_i)^{p_i}<\infty\},\\
J&=\{i:p_i\leq 0,\,\,\e\log f_i(X_i)=-\infty\,\,\mbox{if $p_i=0$}\,\,\mbox{and}\,\,\e f_i(X_i)^{p_i}=\infty\,\,\mbox{if $p_i<0$}\},\\
J'&=\{i:p_i\leq 0,\,\,\e\log f_i(X_i)>-\infty\,\,\mbox{if $p_i=0$}\,\,\mbox{and}\,\,\e f_i(X_i)^{p_i}<\infty\,\,\mbox{if $p_i<0$}\}.
\end{align*}
Note that $I\cup I'\cup J\cup J'=\{1,\ldots,m\}$ and $I\cup J\neq \emptyset.$ In the case that $P\geq T,$ we have $p_1,\ldots,p_m\geq 1$. This means that $I\neq\emptyset$ and $J=\emptyset=J'$. So 
$$\prod_{i=1}^m(\e f_i(X_i)^{p_i})^{{1}/{p_i}}=\prod_{i\in I}(\e f_i(X_i)^{p_i})^{{1}/{p_i}}\cdot \prod_{i\in I'}(\e f_i(X_i)^{p_i})^{{1}/{p_i}}=\infty,$$
 which clearly gives \eqref{main1}. Suppose that $P\leq T.$ If $J\neq \emptyset,$ noting that $(\e f_i(X_i)^{p_i})^{{1}/{p_i}}=0$ for all $i\in J$, it follows that $$\prod_{i=1}^m(\e f_i(X_i)^{p_i})^{{1}/{p_i}}= \prod_{i\in I\cup I'\cup J'}(\e f_i(X_i)^{p_i})^{{1}/{p_i}}\cdot \prod_{i\in J}(\e f_i(X_i)^{p_i})^{{1}/{p_i}}=0$$ and this yields \eqref{main2}. Suppose that $J=\emptyset.$ Then $I\neq \emptyset$ and $\{1,\ldots,m\}=I\cup I'\cup J'.$ Note that $\e (f_i(X_i)\wedge M)^{p_i}<\infty$ for all $M>0$ and each $i\in I$. Applying the proceeding case \eqref{eq88} to $(f_i\wedge M)_{i\in I}$ and $(f_i)_{i\in I'\cup J'}$ gives
\begin{align*}
\e\prod_{i\in I}(f_i(X_i)\wedge M)\cdot\prod_{i\in I'\cup J'}f_i(X_i)
&\geq\prod_{i\in I}\bigl(\e (f_i(X_i)\wedge M)^{p_i}\bigr)^{\frac{1}{p_i}}
\cdot \prod_{i\in I'\cup J'}\bigl(\e f_i(X_i)^{p_i}\bigr)^{\frac{1}{p_i}}.
\end{align*}
From the monotone convergence theorem, letting $M\uparrow \infty$ leads to  
$$
\prod_{i=1}^m\e f_i(X_i)\geq \prod_{i\in I}\bigl(\e f_i(X_i)^{p_i}\bigr)^{\frac{1}{p_i}}\cdot\prod_{i\in I'\cup J'}\bigl(\e f_i(X_i)^{p_i}\bigr)^{\frac{1}{p_i}}=\infty,
$$
which gives \eqref{main2}. This completes our proof.
\end{proof}

\subsection{Second proof: the Ornstein-Uhlenbeck semigroup}\label{sec3}
In the second proof, we will adapt the ideas from \cite{BC} and \cite{CLL}. 
As in the first proof, we will continue to assume that $T_{ii}= I_{n_{i}}$ 
for $1\leq i\leq m$. Consider the Ornstein-Uhlenbeck semigroup operator 
$(P_t)_{t\geq 0}$ defined on $f:\IR^n\rightarrow\IR$ as
\begin{align}
\label{ou}
 P_tf(x)=\int_{\IR^n} f(e^{-t}x + \sqrt{1-e^{-2t}}y)d\gamma_n(y)
\end{align}
with generator
$$
 L=\Delta -\inner{id_n}{\nabla} .
$$
From the definitions of $P_t$ and $L$, we have
\begin{itemize}
 \item[$(P1)$]   $P_tf \rightarrow \int_{\IR^n}f\,d\gamma$ a.s. as $t\rightarrow \infty$,
 \item[$(P2)$]  $P_0f =f$
\end{itemize}
and the integration by parts formula
\begin{equation}\label{eq.IBP}
 \int_{\IR^n} Lfg\,d\gamma_n =
               -\int_{\IR^n} \inner{\nabla f}{\nabla g}\,d\gamma_n .
\end{equation}

\noindent Moreover, the $g(t,x)=P_tf(x)$ satisfies the PDE
\begin{equation*}\label{eq.wave}
 \frac{\partial g}{\partial t}(t,x)=\Delta g(t,x) -\inner{x}{\nabla g(t,x)}=Lg(t,x),
\end{equation*}
and $F^{(t)}(x)=F(t,x):=\log P_tf(x)$ satisfies
\begin{equation}\label{eq.logwave}
 \frac{\partial F^{(t)}}{\partial t}(x)=LF^{(t)}(x) + |\nabla F^{(t)}(x)|^2,
\end{equation}
where $|x|$ stands for the Euclidean norm of the vector $x$. Our first goal is to prove
the following

\begin{theorem}\label{thm.O-U}
  Let $m \geq n$ and $n_1,\ldots,n_m \leq n$ be positive integers and set 
  $N=n_1+\cdots+n_m$. For every $i=1,\ldots,m,$ consider the $n_i\times n$ 
  matrices ${U}_i$ with ${ U}_i\,{ U}_i^{*}=I_{n_i}$. Set the $N\times n$ 
  matrix $U$ consisting of block rows $U_1,\ldots,U_m$ and the $N\times N$ 
  diagonal matrix $D$ with nonzero entries, 
  $$
  D={\rm diag}\Big(d_1I_{n_1},\ldots,d_mI_{n_m}\Big).
  $$
 For nonnegative Lebesgue measurable functions $f_i$ on $\mathbb{R}^{n_i}$ 
 for $1\leq i\leq m$, we have
 \begin{itemize}
  \item[$(i)$] if $UU^*\leq D^{-1}$, then
             \begin{equation}\label{eq.MultiO-U.Up}
               \int_{\IR^n} \prod_{i=1}^mf_i({ U}_ix)\,d\gamma_n(x)
               \leq
               \prod_{i=1}^m \biggl(\int_{\IR^{n_i}} f_i(x_i)^{1/d_i}\,
               d\gamma_{n_i}(x_i)\biggr)^{d_i};
             \end{equation}  
 \item[$(ii)$] if $UU^*\geq D^{-1}$, then
              \begin{equation}\label{eq.MultiO-U.Down}
               \int_{\IR^n} \prod_{i=1}^m f_i({  U}_ix)\,d\gamma_{n}(x)
               \geq
               \prod_{i=1}^m \biggl(\int_{\IR^{n_i}} f_i(x_i)^{1/d_i}\,
               d\gamma_{n_i}(x_i)\biggr)^{d_i}.
              \end{equation}
 \end{itemize}
\end{theorem}
\begin{proof}
As we have discussed in the first proof of Theorem \ref{main} or referring to the approximation procedure in \cite{BC}, we may assume without loss of generality that $f_i$'s are smooth and uniformly bounded from above and away from zero on $\mathbb{R}^n$. For $1\leq i\leq m$, set
$$
F_i^{(t)}(x_i)= \log P_t f_i(x_i),\quad x_i\in\IR^{n_i}.
$$
For $t\in[0,\infty),$ we consider
\begin{align*}
 a(t)&:=\int_{\IR^n}\prod_{i=1}^m P_tf_i({  U}_ix)^{d_i}\,d\gamma_n(x)
     =\int_{\IR^n}
       \exp\biggl(\sum_{i=1}^m d_i F_i^{(t)}({ U}_ix)\biggr)\,d\gamma_n(x).
\end{align*}
Note that from $(P1)$ and $(P2),$
\begin{align*}
\lim_{t\rightarrow\infty} a(t)
&=\prod_{i=1}^m\left(\int_{\IR^{n_i}}f_i\,d\gamma_{n_i}\right)^{d_i}
\,\,\mbox{and}\,\,
a(0)=\int_{\IR^n}\prod_{i=1}^m f_i({ U}_ix)^{d_i}\,d\gamma_n(x).
\end{align*}
Thus, it is enough to show that, under condition $UU^*\leq D^{-1}$
(resp. $UU^*\geq D^{-1}$), $a(t)$ is increasing (resp. decreasing).
To do so, we compute its derivative:
\begin{align*}
a'(t)&=\int_{\IR^n} \sum_{i=1}^m d_i
      \bigl( LF_i^{(t)}({  U}_ix)+\big|\nabla F_i^{(t)}({  U}_ix)\big|^2\bigr)\\
      &\qquad\qquad\qquad
      \exp\biggl(\sum_{i=1}^m d_iF_i^{(t)}({ U}_ix)\biggr)\,d\gamma_n(x),
\end{align*}
where we used (\ref{eq.logwave}) in dimension $n_i$ for all $i=1,\ldots,m$.
Let $t$ be fixed. Set $F_i=F_i^{(t)}$ and
$H_i=F_i\circ{U}_i.$ Since ${U}_i{ U}_i^*=I_{n_i}$,
we have that
\begin{align*}
 LH_i(x)&= \Delta H_i(x) - \inner{x}{\nabla H_i(x)}\\
        &= \Delta F_i({U}_ix) - \inner{{U}_ix}{\nabla F_i({ U}_ix)}\\
        &= LF_i({U}_ix)  .
\end{align*}
Thus, we can use the $n$-dimensional integration by parts formula (\ref{eq.IBP}) 
for the functions $H_i(x)$ and $G(x):=\exp\bigl(\sum_{i=1}^m d_i F_i({U}_ix)\bigr)$ 
to get that for every $1\leq i\leq m,$
\begin{align*}
 &\int_{\IR^n}LF_i({U}_ix)\,G(x)\,d\gamma_n(x)\\
 &=\int_{\IR^n}LH_i(x)\,G(x)\,d\gamma_n(x)\\
  &= -\int_{\IR^n} \inner{\nabla H_i(x)}{\nabla G(x)}\,d\gamma_n(x)\\
  &= -\int_{\IR^n}\sum_{j=1}^m d_j\,\Big\langle\nabla F_i({ U}_ix)
           { U}_i{U}_j^*\nabla F_j({U}_jx)\Big\rangle\exp
           \biggl(\sum_{i=1}^m d_i F_i({  U}_ix)\biggr)d\gamma_n(x).
\end{align*}

\noindent It follows that
\begin{align*}
 a'(t)&=\int_{\IR^n} \biggl(
       -\sum_{i=1}^m\sum_{j=1}^m d_id_j
       \left<\nabla F_i({ U}_ix),
       { U}_i{ U}_j^*\nabla F_j({ U}_jx)\right> \\
       &\qquad\qquad\biggl.+\sum_{i=1}^m d_i \big|\nabla F_i({ U}_ix)\big|^2 \biggr)
       \exp\biggl(\sum_{i=1}^m d_i F_i({ U}_ix)\biggr)\,d\gamma_n(x)
\end{align*}
or equivalently,
\begin{align*}\label{eq.deriv}
 a'(t)&=\int_{\IR^n} \biggl(-\bigg|\sum_{i=1}^m d_i{U}_i^*\nabla F_i({ U}_ix)\bigg|^2
       +\sum_{i=1}^m d_i\big|\nabla F_i({ U}_ix)\big|^2 \biggr)\\
       & \qquad\qquad\qquad\exp\biggl(\sum_{i=1}^m d_i F_i({  U}_ix)\biggr)d\gamma_n(x) .
\end{align*}

\noindent This implies that the proof will be complete if we show that
\begin{itemize}
 \item[$(i)$]  $UU^* \leq D^{-1}$ if and only if
             \begin{equation}\label{eq.l1}
               \bigg|\sum_{i=1}^m d_i\,{U}_i^*\xi_i\bigg|^2
               \leq \sum_{i=1}^m d_i \big|\xi_i\big|^2 ,
               \,\,\forall\;\xi_i\in \IR^{n_i} .
              \end{equation}
 \item[$(ii)$] $UU^* \geq D^{-1}$ if and only if
             \begin{equation}\label{eq.l2}
               \bigg|\sum_{i=1}^m d_i\,{ U}_i^*\xi_i\bigg|^2
               \geq \sum_{i=1}^m d_i \big|\xi_i\big|^2 ,
               \,\,\forall\;\xi_i\in \IR^{n_i} .
              \end{equation}
\end{itemize}
To check \eqref{eq.l2}, we write $\xi=(\xi_1,\ldots,\xi_m)\in\IR^N$ with
$\xi_i\in\IR^{n_i}$, and then we have that
\begin{align*}
 UU^* \geq D^{-1}
 &\Leftrightarrow\inner{UU^*x}{x} \geq \inner{D^{-1}x}{x},\,\, \forall\,x\in\IR^N\\
 \big(x=D\xi\big)
 &\Leftrightarrow \inner{UU^*D\xi}{D\xi} \geq \inner{\xi}{D\xi},\,\,\forall\,\xi\in\IR^N\\
 &\Leftrightarrow |U^*D\xi|^2 \geq \inner{\xi}{D\xi},\,\,\forall\,\xi\in\IR^N\\
 &\Leftrightarrow \bigg|\sum_{i=1}^m d_i\,{ U}_i^*\xi_i\bigg|^2
                   \geq \sum_{i=1}^m d_i \big|\xi_i\big|^2,\,\,\forall\;\xi_i\in \IR^{n_i}.
\end{align*}
The verification of \eqref{eq.l1} is identical.
\end{proof}

\noindent Next, we will show how Theorem \ref{main} can be obtained from Theorem
\ref{thm.O-U}. First, we need a standard linear algebra fact.

\begin{lemma}\label{lemma.Umatrix}
 Let $n,N$ be positive integers. Let $T$ be a $N\times N$ symmetric and positive
 semi-definite matrix with ${\rm rank}(T)=n$. Then there exists a
 $N\times n$ matrix $U=U(T)$ with ${\rm rank}(U)={\rm rank}(T)=n$ such that
 $T=UU^{*}$. Moreover, $U$ is unique up to an orthogonal transformation.
\end{lemma}
\begin{proof}
 The existence of $U$ is guaranteed from the singular value decomposition of $T$.
 More precisely, let us order the eigenvalues $\lambda_1,\ldots,\lambda_N$ of
 $T^2=T^*T=TT^*$ such that 
 $\lambda_1\geq\cdots\geq\lambda_n>0=\lambda_{n+1}=\cdots=\lambda_N$
 and let $v_1,\ldots,v_N$ be the corresponding
 eigenvectors. Consider the $N\times N$ matrices $V=(v_1,\ldots,v_N)$ and
 $L={\rm diag}(\lambda_1,\ldots,\lambda_n,0,\ldots,0)$. Then from the singular value
 decomposition, we have that
 $$
 T=VLV^{*}=(V\sqrt{L})(V\sqrt{L})^{*}=:V_L\,V_L^{*} ,
 $$
 where $V_L:=V\sqrt{L}$. We write
 \begin{align*}
 V_L &=
 \left(
\begin{array}{ccc}
     v_1    & \ldots &  v_N  
\end{array}
\right)
{\rm diag}(\sqrt{\lambda_1},\ldots,\sqrt{\lambda_n},0,\ldots,0) \\
&=\left(
\begin{array}{cccccc}
 \sqrt{\lambda_1}v_1 & \ldots &  \sqrt{\lambda_n}v_n &  
 \mathbb{O}_{N\times 1}&\ldots&\mathbb{O}_{N\times 1} 
\end{array}
\right) \\
&=\left(
\begin{array}{cc}
 u_1   &  \mathbb{O}_{1\times(N-n)}   \\
\vdots &  \vdots \\
 u_N   & \mathbb{O}_{1\times(N-n)}  \\
\end{array}
\right)
=:\left(
\begin{array}{cc}
 U   &  \mathbb{O}_{N\times(N-n)} 
\end{array}
\right),
\end{align*}
and so
\begin{equation*}
 T=V_L\,V_L^{*}=\Bigl(\inner{u_i}{u_j}\Bigr)
 = U\,U^{*} ,
\end{equation*}
where $U$ is the $N\times n$ matrix with rows $u_1,\ldots,u_N$.

For the uniqueness of $U$, we need to show that if $V$ is a $N\times n$ 
matrix with $VV^*=T=UU^*$, then $\Phi U^*=V^*$ for some $\Phi \in O(n)$, 
orthogonal transformation in $\IR^n$. If we write $v_1,\ldots,v_N$ for the 
rows of $V$ we have that $\IR^n={\rm span}\{u_1,\ldots,u_N\}={\rm span}\{v_1,\ldots,v_N\}$.
Define the linear transformation $\Phi:\IR^n\rightarrow\IR^n$ such that $\Phi u_i={ v}_i$
for all $i=1,\ldots,N$ or equivalently $\Phi U^*=V^*$. With this construction, one clearly
sees that $\Phi\in O(n)$. Indeed, by definition,
$$
\inner{\Phi u_i}{\Phi u_j}=\inner{{v}_i}{{v}_j}=\inner{u_i}{u_j}
$$
for  $1\leq i,j\leq N$ and so
\begin{eqnarray*}
 \inner{\Phi x}{\Phi x}
 = \sum_{i=1}^N\sum_{j=1}^N a_i a_j\inner{\Phi u_i}{\Phi u_j}
 = \sum_{i=1}^N\sum_{j=1}^N a_i a_j\inner{u_i}{u_j}
 = \inner{x}{x}
\end{eqnarray*}
for every $x=\sum_{i=1}^N a_i u_i\in\mathbb{R}^n$. This completes our proof.
\end{proof}

\noindent We are now ready to complete the second proof of our main result:

\begin{proof}[Second proof of Theorem \ref{main}]
 Without loss of generality we can assume that $p_i$'s are non-zero. 
 Recall that we have assumed that $T_{11}=I_{n_1},\ldots,T_{mm}=I_{n_m}.$ 
 Let $n={\rm rank}(T).$ From Lemma \ref{lemma.Umatrix}, there exists a 
 $N\times n$ matrix $U$ such that $T=UU^*$. We denote by 
 $u_1^i,\ldots,u_{n_i}^i$ the rows of $U_i$ and by $U$ the $N\times n$ 
 matrix with block rows $U_1,\ldots,U_m$. Since
  $$
  T=UU^*=({U}_i{ U}_j^*)_{i,j\leq m},
  $$
  we have that ${ U}_i{ U}_i^*=T_{ii}=I_{n_i}$ for $1\leq i\leq m$.
  On the other hand, observe that $\big(X_{ij}:1\leq i\leq m,1\leq j\leq n_i\big)$
  and $\big(\inner{Z}{u^i_j}:1\leq i\leq m,1\leq j\leq n_i\big)$ are identically
  distributed, where
  ${Z}$ is a $n$-dimensional standard Gaussian random vector. So
  $$
  {\bf X}=
  \left(\begin{array}{c}
  { X}_1\\
  \vdots\\
  { X}_m
  \end{array}\right)
  \stackrel{d}{=}
   \left(\begin{array}{c}
  { U}_1{Z}\\
  \vdots\\
  { U}_m{Z}
  \end{array}\right)
  =U{Z}.
  $$
  Thus, we have that
  $$
  \IE \prod_{i=1}^m f_i({ X}_i) = \IE \prod_{i=1}^m f_i({ U}_i{ Z})
  = \int_{\IR^n} \prod_{i=1}^mf_i({U}_ix)d\gamma_n(x)
  $$
  and Theorem \ref{main} follows immediately from Theorem \ref{thm.O-U}.
\end{proof}

Actually, it is easy to show that Theorem \ref{main} implies Theorem
\ref{thm.O-U}. Indeed, if $U$ and $D$ are as in Theorem \ref{thm.O-U}, then
$T=UU^*$ and $P=D^{-1}$ satisfy the assumptions of Theorem \ref{main}.
Working as in the previous proof, Theorem \ref{thm.O-U} follows.

\medskip

\section{The Brascamp-Lieb inequality, the geometry of eligible exponents
          and Gaussian hypercontractivity}\label{sec4}

This section will be concentrated on Theorem \ref{main}(i). The equivalence between 
this bound and the geometric form of the Brascamp-Lieb inequality for Gaussian measures 
will first be established. Next, we turn to the study of some geometric properties of 
the eligible exponents in Theorem \ref{main}(i). We close this section by showing that 
theorem \ref{main} generalize the Gaussian hypercontractivity and its reverse form.

\subsection{Connection to the Brascamp-Lieb inequality}

The main objective of this subsection is to show that Theorem \ref{main}(i) is 
a reformulation of the geometric Brascamp-Lieb inequality for Gaussian 
measures, which is stated below.

\begin{theorem}\label{G-BL}
 Assume that $n \leq m$ and $n_1,\ldots,n_m \leq n$ are positive integers. For every
 $i=1,\ldots,m$, consider the $n_i\times n$ matrices ${ U}_i$ with
 ${ U}_i\,{ U}_i^{*}=I_{n_i}$ and $p_i >0$ such that
 \begin{equation}\label{eq.John2}
  U^*P^{-1}U = I_n,
 \end{equation}
 where $P:=\mbox{diag}(p_1I_{n_1},\ldots,p_m I_{n_m}).$
 Then for measurable function $f_i:\IR^{n_i}\rightarrow [0,\infty)$ for $i=1,\ldots,m,$
 one has that
 \begin{equation}\label{eq.G-BL}
  \int_{\IR^n} \prod_{i=1}^m f_i({ U}_ix)\,d\gamma_n(x)
  \leq \prod_{i=1}^m\|f_i\|_{L_{p_i}(\gamma_{n_i})},
\end{equation}
where the notation $\gamma_k$ means the $k$-dimensional standard Gaussian measure 
on $\mathbb{R}^k.$
\end{theorem}

 Note that as we have seen from Subsection \ref{sec3}, Theorem \ref{thm.O-U}(i) is 
 equivalent to Theorem \ref{main}(i). To attend our goal, it suffices to establish 
 the equivalence between Theorem \ref{thm.O-U}(i) and Theorem \ref{G-BL}. The argument 
 that the second implies the first is simple. Indeed, assuming that the assumptions 
 $U_iU_i^*=I_{n_i}$ for $i\leq m$ and \eqref{eq.G-BL} holds for some $p_1,\ldots,p_m>0$, 
 one sees that
 \begin{align}
 \begin{split} \label{arg}
  U^*P^{-1}U = I_n &\Rightarrow U^*P^{-1}U \leq I_n \\
  &\Leftrightarrow \lambda_1(U^*P^{-1}U) \leq 1 \\
  &\Leftrightarrow \lambda_1(P^{-1/2}UU^*P^{-1/2}) \leq 1 \\
  &\Leftrightarrow P^{-1/2}UU^*P^{-1/2} \leq I_N \\
  &\Leftrightarrow UU^* \leq P,
 \end{split}
 \end{align}
where $\lambda_1(A):=\norm{A}_{OP}$, the largest eigenvalue of the real symmetric 
matrix $A$. Consequently, the assumptions of Theorem \ref{thm.O-U}(i) are satisfied 
by $U_i$'s and $d_i:=p_i^{-1}$ for $i\leq m$ and thus \eqref{eq.G-BL} follows from 
\eqref{eq.MultiO-U.Up}.

\smallskip

As for the reverse direction, recall $U$ and $D$ from Theorem \ref{thm.O-U}(i) and
set $P=D^{-1}.$ Then $UU^*\leq P$ is equivalent to $\|A\|_{OP}\leq 1$ 
for $A:=U^*P^{-1}U.$ Let $\lambda_1,\ldots,\lambda_n\geq 0$ be the eigenvalues 
of $A$ listed in non-increasing order and $\theta_1\ldots,\theta_n$ be the 
corresponding orthonormal eigenvectors. Let $k$ be the largest integer such 
that $\lambda_1=\cdots=\lambda_k.$ Consider the decomposition of the identity 
matrix $I_n$,
 \begin{align} \label{decompId}
 \sum_{i=1}^m \frac{1}{p_i\lambda_1} U_i^* U_i+\sum_{i=k+1}^n 
 \left(1-\frac{\lambda_i}{\lambda_1}\right) \theta_i\theta_i^*=I_n,
 \end{align}
 where the validity of this identity can be easily checked by showing that both 
 sides agree on $\theta_1,\ldots,\theta_n.$ 
 If $\lambda_1<1,$ this equation may as well be written as
  \begin{align*}
  \sum_{i=1}^m \frac{1}{p_i} U_i^* U_i+\sum_{i=k+1}^n 
  \left(1-\frac{\lambda_i}{\lambda_1}\right) \theta_i\theta_i^* +
  \sum_{i=1}^m \frac{1}{p_i}\left(\frac{1}{\lambda_1}-1\right) U_i^* U_i=I_n.
  \end{align*}
 Note that the coefficient terms in the last two equations are all positive since 
 $\lambda_1=\|A\|_{OP}\leq 1$. To sum up, there exists  some $\nu\in \mathbb{N}\cup\{0\}$ 
 such that there are $k_j\times n$ matrix ${B}_j$ with ${B}_j{B}_j^*=I_{k_j}$ and $b_j>0$ 
 for $1\leq j\leq \nu$ satisfying
  \begin{align}\label{sec4:eq1}
  \sum_{i=1}^m \frac{1}{p_i}{ U}_i^*U_i
         +\sum_{j=1}^{\nu} \frac{1}{ b_j}B_j^*B_j=I_n,
  \end{align}
 where $\sum_{j=1}^{0}b_j^{-1}B_j^*B_j$ is read as the $n$-dimensional zero matrix.
 For given nonnegative measurable function $f_i$ on $\mathbb{R}^{n_i}$ for $i\leq m,$ 
 we set $g_1=f_1,\ldots,g_m=f_m,g_{m+1}=1,\ldots,g_{m+\nu}=1.$ Since $p_i$'s and 
 $b_j$'s satisfy \eqref{sec4:eq1}, we may apply these $g_i$'s to \eqref{eq.G-BL} 
 to obtain \eqref{eq.l1} by noting that $\|g_i\|_{L_{b_i}(\gamma_{k_i})}=1$ for all 
 $m+1\leq i\leq m+\nu.$ This completes our argument.

\subsection{The geometry of the eligible exponents.} 

Let $n\leq N$ and $U$ be a $N\times n$ matrix with ${\rm rank}(U)=n$.
Assume that $U$ has as block-rows the $n_i\times n$ matrices $U_i$, $1\leq i \leq m$, 
with $U_iU_i^*=I_{n_i}$. Then we define 
\begin{equation}\label{def.C(T)}
 \mathcal{C}(U)=\left\{(c_1,\ldots,c_m):\;UU^* \leq C^{-1}\right\}
\end{equation}
where $C={\rm diag}(c_1I_{n_1},\ldots,c_mI_{n_m})$. By the discussion in the last 
subsection we have that if $(c_1,\ldots,c_m)\in\mathcal{C}(U)$ then $c_1,\ldots,c_m$ 
satisfy \eqref{sec4:eq1}. On the other hand if an m-tuple $c_1,\ldots,c_m$ satisfies
\eqref{sec4:eq1} then trivially, $U^*CU=\sum_{i=1}^mc_iU_i^*U_i \leq I_n$ and by 
\eqref{arg} this means that $(c_1,\ldots,c_m)\in\mathcal{C}(U)$. Thus, we have 
proved that 
\begin{equation}
 (c_1,\ldots,c_m)\in\mathcal{C}(U)\quad
 \text{if and only if}\quad (c_1,\ldots,c_m)\;\text{satisfies \eqref{sec4:eq1}}
\end{equation}

Next we gather some interesting properties for the set $\mathcal{C}(U)$.

\begin{proposition}\label{lemma.C(T)}
 Let $U$ be the $N\times n$ matrix defined in Definition \ref{def.C(T)}.
 \begin{itemize}
 \item[$(i)$]   Let $V$ be a matrix with the same size as $U$ and $UU^*=VV^*$, 
                then $\mathcal{C}(U)=\mathcal{C}(V)$.
 \item[$(ii)$]  Define the $n_{i} \times N$ matrix 
                $R_{i}= \left({\bf 0}_{n_i\times n_1},
                \ldots, I_{n_{i}}, \ldots, {\bf 0}_{n_i\times n_m}\right)$,
                for $1\leq i\leq m$. For $\sigma\subseteq \{1, \ldots, m\}$, 
                set the $\left(\sum_{i\in \sigma} n_{i}\right)\times N $ matrix
                $\Xi_{\sigma}:= \left([ R_{i}^*]; i\in\sigma\right)^{*},$ i.e. 
                the matrices
                $R_i$, $i\in\sigma$ are the block-rows of $\Xi_\sigma$. Then
                $$
                P_{\sigma} \mathcal{C}(U) \subseteq \mathcal{C}( \Xi_{\sigma} U),
                $$
                where $P_{\sigma}$ denotes the projection from $\mathbb{R}^m$ to 
                $\mathbb{R}^{\sigma}$ through $P_\sigma(c)=(c_i)_{i\in\sigma}$.
 \item[$(iii)$] $C(U)$ is a convex subset of $\IR^m$ and
                $$
                \biggl\{ x\in \mathbb [0,\infty)^{m}:
                \sum_{i=1}^{m} x_{i} \leq 1\biggr\} 
                \subseteq \mathcal{C}(U) \subseteq [0,1]^{m}.
                $$
 \item[$(iv)$]  If $(c_{1}, \ldots, c_{m}) \in \mathcal{C}(U)$ and 
                $\lambda_{1},\ldots,\lambda_{m}\in[0,1]$, then
                \begin{equation}\label{eq.kouti}
                  (\lambda_{1} c_{1}, \ldots, \lambda_{m} c_{m}) \in \mathcal{C}(U).
                \end{equation}             
 \end{itemize}
\end{proposition}

\begin{proof} $(i)$ From Lemma \ref{lemma.Umatrix}, we have that $V^*=\Phi U^*$ 
  for some $\Phi\in O(n)$. Let $(c_1,\ldots,c_m)\in\mathcal{C}(U)$. By the definition 
  of $\mathcal{C}(U)$, this means that for some
  $$
  {\rm B}=\left(\begin{array}{c}
   \left[\leftarrow \mathcal{B}_1 \rightarrow\right]  \\
    \vdots                               \\
   \left[\leftarrow\mathcal{B}_\nu\rightarrow\right]    \\
  \end{array}\right)
  $$
  and
  $
  L={\rm diag}(b_1I_{k_1},\ldots,b_\nu I_{k_\nu}),
  $
  we have that
  $
  U^*C\,U + {\rm B}^*L\,{\rm B} = I_n.
  $
  Taking $\Phi$ and $\Phi^*$, we write equivalently
  $$
  \Phi U^*C\,U \Phi^* + \Phi{\rm B}^*L\,{\rm B}\Phi^* = \Phi I_n \Phi^*
  $$
  or
  $$
  V^*C\,V + {\rm B}_\Phi^*L\,{\rm B}_\Phi = I_n,
  $$
  where
  $$
  {\rm B}_\Phi:={\rm B}\Phi^*=\left(\begin{array}{c}
  \left[\leftarrow \mathcal{B}_1\Phi^* \rightarrow\right]  \\
     \vdots                               \\
  \left[\leftarrow\mathcal{B}_\nu\Phi^*\rightarrow\right]    \\
  \end{array}\right).
  $$
  This gives that $(c_1,\ldots,c_N)\in\mathcal{C}(V)$. Thus, we have proved that
  $\mathcal{C}(U)\subseteq\mathcal{C}(V)$. The same argument gives also the other 
  inclusion and the claim follows.

  \smallskip

  As for $(ii)$, let $x=(x_i)_{i\in\sigma}\in P_{\sigma} \mathcal{C}(U)$. This means
 that for some $c=(c_1,\ldots,c_m)\in\mathcal{C}(U)$, we have that
 $x_i=c_i$ for all $i \in\sigma$. Recall from the definition of $\mathcal{C}(U)$ that
 the equation \eqref{sec4:eq1} holds and it can be rewritten as
 $$
 \sum_{i\in\sigma} x_i\,\mathcal{U}_i^*\mathcal{U}_i +
 \sum_{i\notin\sigma} c_i\,\mathcal{U}_i^*\mathcal{U}_i +
 \sum_{j=1}^{\nu} b_j\,\mathcal{B}_j^*\mathcal{B}_j = I_n.
 $$
 Note that $\Xi_\sigma U$ has as block rows, the matrices $\mathcal{U}_i$
 for $i\in\sigma$. The last equation guarantees that $x\in\mathcal{C}(\Xi_\sigma U)$.

 \smallskip

 For $(iii)$, assume that $(c_{1},\ldots,c_{m})$, $(\hat{c}_{1},\ldots,\hat{c}_{m})
 \in C(U)$ and $\lambda\in [0,1]$. Then there exist $b_{j}$'s, $\hat{b}_{j}$'s,
 $\mathcal{B}_j$'s and $\hat{\mathcal{B}}_j$'s such that
 $$
 \sum_{i=1}^m {c}_i\,\mathcal{U}_i^*\mathcal{U}_i +
 \sum_{j=1}^{\nu_{1}} b_j\,\mathcal{B}_j^*\mathcal{B}_j=I_n\,\,{\rm and}\,\,
 \sum_{i=1}^m \hat{c}_{i}\mathcal{U}_i^*\mathcal{U}_i + \sum_{j=1}
 ^{\nu_{2}} \hat{b}_j\,\hat{\mathcal{B}}_{j}^*\hat{\mathcal{B}}_j=I_n.
 $$
 Consequently,
 \begin{eqnarray*}
   &\sum_{i=1}^m (\lambda c_i+ (1-\lambda)\hat{c}_{i})\,
   U_i^* U_i + \sum_{j=1}^{\nu_{1}}\lambda b_j\,
   B_j^* B_j + \sum_{j=1}^{\nu_{2}} (1-\lambda)
    \hat{b}_{j}^{\prime}\hat{B}_j^*\hat{B}_j\\
   &=
   \lambda\biggl(\sum_{i=1}^m c_i U_i^* U_i
   + \sum_{j=1}^{\nu_{1}} b_j B_j^* B_j\biggl)
   +
   (1-\lambda)\biggl(\sum_{i=1}^m \hat{c}_i\, U_i^* U_i
   + \sum_{j=1}^{\nu_{2}} \hat{b}_j\hat{B}_j^*\hat{B}_j\biggr) \\
   &=
   \lambda I_n + (1-\lambda)I_n\\
   & = I_n,
 \end{eqnarray*}
 which means that $\lambda c_i+ (1-\lambda)\hat{c}_{i}\in \mathcal{C}(U)$ and
 this shows the convexity of $\mathcal{C}(U)$. For the second part of the assertion
 $(iii)$, since $\mathcal{U}_i^*\mathcal{U}_i$ is a projection from $\IR^n$ to 
 $\IR^{n_i}$ for $1\leq i\leq m$, we have that
 \begin{equation}\label{eq.B_1^m}
 \norm{U^*CU}_{OP}=\biggl\|\sum_{i=1}^m c_i\,\mathcal{U}_i^*\mathcal{U}_i\biggr\|_{OP}
 \leq \sum_{i=1}^m c_i \,\norm{\mathcal{U}_i^*\mathcal{U}_i}_{OP}\leq \sum_{i=1}^m c_i
 \end{equation}
 for all $N\times N$ diagonal matrix $C={\rm diag}(c_1I_{n_1},\ldots,c_mI_{m_m}).$
 From  \eqref{arg}, we have that
 $c=(c_1,\ldots,c_m)\in\mathcal{C}(U)\Leftrightarrow\norm{U^*CU}_{OP}\leq 1$,
 and so, by \eqref{eq.B_1^m}
 $$
 \biggl\{ x\in \mathbb [0,\infty)^{m}:
 \sum_{i=1}^{m} x_{i} \leq 1\biggr\}\subseteq \mathcal{C}(U).
 $$
 On the other hand, from \eqref{arg} again, we have that
 $$
 (c_1,\ldots,c_m)\in\mathcal{C}(U)\Leftrightarrow UU^*-C^{-1}\leq 0
 \Leftrightarrow\left<(UU^*-C^{-1}),x,x\right>\leq 0,\,\,\forall x\in\mathbb{R}^N.
 $$
 Taking the vectors ${\bf x}_i=(0,\ldots,0,x_i,0,\ldots,0)\in\IR^N$, $1\leq i\leq m$, 
 for any non-zero $x_i\in\IR^{n_i}$, we get that $c_i\leq 1$ 
 and this shows that $\mathcal{C}(U)\subset [0,1]^m$.
 
 \smallskip
 
 Finally, $(iv)$ can be be easily verified by rewriting the equation \eqref{sec4:eq1} as
\[
\sum_{i=1}^m \lambda_{i} c_i U_i^* U_i + \sum_{i=1}^m (1-\lambda_{i})c_i
U_i^* U_i + \sum_{j=1}^\nu b_j B_j^*  B_j = I_n.
\]
 \end{proof}

 We are now ready to discuss the geometry of the eligible exponents in Theorem 
 \ref{main}(i). Note that we consider only its normalized version, i.e. we assume 
 that $T_{ii}=I_{n_i}$ for every $1\leq i\leq m$. Nevertheless, as one could see by
 a simple change of variables, this simpler version of Theorem \ref{main}, is just 
 an equivalent reformulation of the initial general statement.
 
 \smallskip
 
 Let $\mathbf{X}$ be the Gaussian random vector in $\IR^N$, with covariance matrix 
 $T=(T_{ij})_{i,j\leq m}$, as in Theorem \ref{main}, with $T_{ii}=I_{n_i}$. 
 We define $\mathcal{C}(\mathbf{X})$ in $\IR^m$ to be the 
 set of all vectors $(1/p_1,\ldots,1/p_m)\in \mathbb [0,\infty)^m$ such that
 \begin{equation}\label{eq.C(X)}
 \IE \prod_{j=1}^m f_j(X_j)\leq \prod_{i=1}^m  \|f_i\|_{L^{p_i}(\gamma_{n_i})},
 \quad\forall\;f_i,
 \;\;1\leq i\leq m.
 \end{equation}
 We also define $\mathcal{B}(\mathbf{X})$ as the set of all vectors 
 $(1/p_1,\ldots,1/p_m)\in \mathcal{C}(\mathbf{X})\cap (0,1)^m$, with the following 
 property: For every $1\leq i\leq m$, if $q>p_i$, then there exist $f_1,\ldots,f_m$ 
 measurable functions, such that
 $$
 \IE\prod_{i=1}^m f_i(X_i) >
 \prod_{j\neq i} 
 \left(\IE f_j(X_j)^{p_j}\right)^{1/p_j}\; \left(\IE f_i(X_i)^q\right)^{1/q}.
 $$
 If $(1/{p_{1}}, \ldots, 1/{p_{m}}) \in \mathcal{C}(\mathbf{X})$, we say that
 $(p_{1},\ldots, p_{m})$ are \textit{eligible exponents} in Theorem \ref{main}(i).
 Respectively, if $(1/{p_{1}}, \ldots, 1/{p_{m}}) \in \mathcal{B}(\mathbf{X})$, we 
 say that $(p_{1},\ldots, p_{m})$ is  a choice of \textit{best possible exponents} 
 in Theorem \ref{main}(i).
 By Lemma \ref{lemma.Umatrix}, there exists a matrix $U$ such that $T=UU^*$ and
 we set $\mathcal{C}(T)=\mathcal{C}(U)$. Observe that $\mathcal{C}(T)$ is 
 well-defined by Proposition \ref{lemma.C(T)}(i). Finally, by Remark \ref{rmrk2} 
 we have that $\mathcal{C}(\mathbf{X})=\mathcal{C}(U)=\mathcal{C}(T)$. Moreover, 
 we know that
 $\mathcal{C}(\mathbf{X})$ is a convex set of $[0,\infty)^m$ that satisfies
 $$
 \biggl\{ y \in [0,\infty)^m:\sum_{i=1}^m y_i 
 \leq 1\biggr\}\subseteq\mathcal{C}(\mathbf{X})\subseteq(0,1]^m.
 $$
 Since $\mathcal{C}(\mathbf{X})$ has the property \eqref{eq.kouti}, it can be 
 extended to an $1$-unconditional convex body $\tilde{\mathcal{C}}(\mathbf{X})$  
 in $\IR^m$ in the obvious way: 
 $(c_{1}, \ldots, c_{m}) \in \tilde{\mathcal{C}}(\mathbf{X})$ if and only if
 $ (|c_{1}|, \ldots, |c_{m}|) \in \mathcal{C}(\mathbf{X})$. In this case we have 
 that
 \begin{equation}
  B_1^m \subseteq \tilde{\mathcal{C}}(\mathbf{X}) \subseteq B_\infty^m,
 \end{equation}
 where
 \begin{align*}
 B_1^m=\{x\in\IR^m:\sum_{i\leq m} |x_i| \leq 1\}
 \quad{\rm and}\quad
 B_{\infty}^m=\{x\in\IR^m:\max_{i\leq m}|x_i| \leq 1\}.
 \end{align*}
The associated norm in $\IR^m$, is given by
 $$
 \norm{c}_{\tilde{\mathcal{C}}(X)}=\norm{U\,|{\rm C}|\,U^*}_{op}
 $$
 for every $c=(c_1,\ldots,c_m)\in\IR^m$, where
 $|{\rm C}|:={\rm diag}\big(|c_1|I_{n_1},\ldots,|c_m|I_{n_m}\big)$.
 Moreover, one can show that if 
 $(1/p_1,\ldots,1/p_m)\in\partial\,\mathcal{C}(\mathbf{X})\cap (0,1)^m$, 
 then there exists an ${\rm a}=({\rm a}_1,\ldots,{\rm a}_m)$ in $\IR^N$
 where ${\rm a}_i 
 \in\IR^{n_i}$,  and
 $f_i:\IR^{n_1}\rightarrow\IR$ of the form $f_i(x_i)=\exp(\inner{{\rm a}_i}{x_i})$,
 $x_i\in\mathbb{R}^{n_i}$, such that equality holds in \eqref{main1}.
 Indeed, by Remark \ref{rmk3}, it is enough to show that
 $\inner{(P-T){\rm a}}{\rm a}=0$. First note that
 \begin{eqnarray*}
  (c_1,\ldots,c_m)\in\partial\,\mathcal{C}(\mathbf{X}) \Leftrightarrow
  \norm{\sqrt{C}\,T\sqrt{C}}_{op}=1 \Leftrightarrow \lambda_1(\sqrt{C}\,T\sqrt{C})=1 .
 \end{eqnarray*}
 Let ${\rm v}_1\in \IR^N$ be the normal eigenvector of $\lambda_1=\lambda_1
 (\sqrt{C}\,T\sqrt{C})$. Then, for ${\rm a}=\sqrt{C}\,{\rm v}_1\in\IR^N$ we have that
 $$
 \inner{T{\rm a}}{\rm a}=\inner{\sqrt{C}\,T\sqrt{C}\,{\rm v}_1}{{\rm v}_1}=
 \lambda_1\inner{{\rm v}_1}{{\rm v}_1}=1=\inner{C^{-1}{\rm a}}{\rm a}.
 $$
 Finally, note that in particular we have shown that 
 $\partial \mathcal{C}(\mathbf{X}) \cap (0,1)^{m}
 \subseteq \mathcal{B}(\mathbf{X})$. Also $\mathcal{B}(\mathbf{X}) \subseteq \partial
 \mathcal{C}(\mathbf{X}) \cap (0,1)^{m}$ by H\"older's inequality, and thus
 we have proved the following
 \begin{proposition}
  Let $\mathbf{X}$ be the Gaussian random vector in $\IR^N$, with covariance matrix 
  $T=(T_{ij})_{i,j\leq m} \neq I_N$ , as in Theorem \ref{main}, with $T_{ii}=I_{n_i}$,
  and $\mathcal{B}(\mathbf{X})$
  as defined above. Then for every $c=(c_1,\ldots,c_m)\in\IR^m$ and
  $C={\rm diag}(c_1I_{n_1},\ldots,c_mI_{n_m})$, we have that
  $$
  c\in\mathcal{B}(\mathbf{X}) \Leftrightarrow c\in\partial\mathcal{C}(\mathbf{X})\cap(0,1)^m
  \Leftrightarrow \norm{\sqrt{C}T\sqrt{C}}_{op}=1.
  $$
  Moreover, for every $c=(c_1,\ldots,c_m)\in\mathcal{B}(\mathbf{X})$, there exist functions
  $f_i(x_i)=\exp(\inner{\alpha_i}{x_i})$, $x_i\in\mathbb{R}^{n_i}$ such that
  one has equality in \eqref{main1} for $(p_1,\ldots,p_m)=(1/c_1,\ldots,1/c_m)$.
 \end{proposition}

 \medskip

 Let us close the discussion, considering again, the simplest case where $m=2$ and
 $n_1=n_2=1$, i.e. $\mathbf{X}=(X_1,X_2)$, where $X_1$ and $X_2$ are two standard Gaussian
 random variables with $\IE X_1X_2=t\in[0,1]$. A direct computation shows that the
 set of all eligible exponents is
 $$
 \mathcal{C}(\mathbf{X})=
 \Big\{(x,y)\in[0,1]^2:\,\Big(\frac1x-1\Big)\Big(\frac1y-1\Big)\geq t^2\Big\}
 $$
and
 $$
 \norm{(x,y)}_{\tilde{\mathcal{C}}(X)}=
 \frac{\sqrt{(|x|+|y|)^2-4(1-t^2)\,|xy|\,}\,+\,|x|+|y|}{2} .
 $$
 Moreover, the couple of exponents $(p_1,p_2)$ with $p_1,p_2\geq 1$ is best possible
 in \eqref{main1} if and only if $(1/p_1,1/p_2)$ lies on
 $\mathcal{B}(\mathbf{X})=\partial\,\mathcal{C}(\mathbf{X})\cap (0,1)^2$ or equivalently, 
 if and only if
 \begin{eqnarray}
 \label{last-eq}
  (p_1-1)(p_2-1) = t^2p_1p_2.
 \end{eqnarray}
 
 
\subsection{Gaussian hypercontractivity inequalities}\label{sec:hyper} 

Recall the Ornstein-Uhlenbeck semi-group operators $(P_t)_{t\geq 0}$ form 
\eqref{ou}. The Gaussian hypercontractivity, discovered by Nelson \cite{Nel2}, 
states that if $p,q>1$ and $t\geq 0$ satisfy $(p-1)(q-1)^{-1}\geq e^{-2t},$ 
then 
\begin{equation}\label{eq.Hyper}
 \norm{{P}_tf}_{L_q(\gamma_n)} \leq \norm{f}_{L_p(\gamma_n)}
\end{equation}
for any measurable $f:\mathbb{R}^n\rightarrow\mathbb{R}$. See also 
\cite{Bec,Bon,BL,Gross} for various approaches to this inequality.
Later, Borell \cite{Bor} proved a reverse hypercontractivity inequality 
for the Bernoulli probability measure. His result was recently extended 
by Mossel, Oleszkiewicz and Sen in \cite{MOS} to a more general class of 
probability measures satisfying log-Sobolev inequalities of certain type. 
In the special case of the Gaussian measure, their result states that if 
$p,q<1$ and $t\geq 0$ with $(1-p)(1-q)^{-1}\geq e^{-2t}$, then
\begin{equation}\label{eq.Rev-Hyper}
\norm{{P}_tf}_{L_q(\gamma_n)} \geq \norm{f}_{L_p(\gamma_n)}
\end{equation}
for any measurable $f$ on $\mathbb{R}^n$. 

\smallskip

In this subsection we show that Theorem \ref{main} generalizes those two 
results. 
To recover \eqref{eq.Hyper} and \eqref{eq.Rev-Hyper} from Theorem \ref{main}, 
consider two $n$-dimensional standard Gaussian random vectors $X$ and $Y$ such 
that their joint law has the $2n\times 2n$ covariance matrix
\begin{equation}\label{eq.T-Hyper}
 T=\left(\begin{array}{cc}
I_n & e^{-t}I_n\\
e^{-t}I_n & I_n
\end{array}\right),\,\,t\geq 0.
\end{equation}
For arbitrary measurable functions $f,g:\mathbb{R}^n\rightarrow\mathbb{R},$
\begin{equation}\label{eq.Gaussian-Perturbation}
 \IE g(X) f(Y)=\IE g(X){P}_tf(X).
\end{equation}
Indeed, note that since Gaussian random vector is characterized by its mean 
and  covariance, $(X,Y)$ has the same joint distribution as 
$(X,e^{-t}X+\sqrt{1-e^{-2t}}Z)$, where $Z$ is an independent copy of $X$. 
A standard computation using conditional expectation yields
\begin{align*}
 \IE g(X) f(Y)&= 
 \IE\biggl(\IE\left.\Bigl(g(X)f(e^{-t}X+\sqrt{1-e^{-2t}}Z)\right|X\Bigr)\biggr)
 = \IE g(X){P}_tf(X).
\end{align*}
For a real number $r\neq 1,$ let $r'$ be the H\"{o}lder conjugate exponent of $r$. 
Let $p,q\in \IR$ and $q\neq 1$, consider the $2n\times 2n$ diagonal matrix,
\begin{equation*}\label{eq.Hyper-diagonal}
 P_o=\left(\begin{array}{cc}
q'I_n      & \mathbb{O}\\
\mathbb{O} & pI_n
\end{array}\right).
\end{equation*}
A direct computation shows that, for any $p,q>1$, 
\begin{equation}\label{eq.UpEquivalence}
 T\leq P_o \Leftrightarrow \frac{p-1}{q-1}\geq e^{-2t}
\end{equation}
and for any $p,q<1$, 
\begin{equation}\label{eq.DownEquivalence}
 T\geq P_o\Leftrightarrow \frac{1-p}{1-q}\geq e^{-2t}.
\end{equation}
Recall the duality relations
\begin{equation}\label{eq.Up-Dual}
\norm{f}_{L_r(\gamma_n)}=\sup_{\norm{g}_{L_{r'}(\gamma_n)}\leq 1}\IE f(X)g(X),\,\,r>1
\end{equation}
and for $f$ nonnegative,
\begin{equation}\label{eq.Down-Dual}
\norm{f}_{L_r(\gamma_n)}=\inf_{g>0,\norm{g}_{L_{r'}(\gamma_n)}\geq 1}\IE f(X)g(X),\,\,r<1.
\end{equation}
Suppose that $f$ is a measurable function on $\mathbb{R}^n.$ If $p,q>1$
with $(p-1)/(q-1)\geq e^{-2t},$ then \eqref{eq.UpEquivalence} implies $T\leq P_o$
and thus from \eqref{eq.Up-Dual} and Theorem \ref{main} (i),
\begin{align*}
 \norm{{ P}_tf}_{L_q(\gamma_n)}
 &= \sup_{\norm{g}_{L_{q'}(\gamma_n)}\leq 1} \IE g(X){P}_tf(X)\\
 &= \sup_{\norm{g}_{L_{q'}(\gamma_n)}\leq 1} \IE g(X)f(Y) \\
 &\leq \sup_{\norm{g}_{L_{q'}(\gamma_n)}\leq 1}
        \norm{g}_{L_{q'}(\gamma_n)}\,\norm{f}_{L_{p}(\gamma_n)} \\
 &= \norm{f}_{L_{p}(\gamma_n)},
\end{align*}
which gives \eqref{eq.Hyper}. Proceeding similarly by using \eqref{eq.DownEquivalence},
\eqref{eq.Down-Dual} and Theorem \ref{main} (ii) yields the reverse inequality
\eqref{eq.Rev-Hyper}

\smallskip

Conversely, one may retrieve the special case of Theorem \ref{main}, for $m=2$, 
$n_1=n_2=n$ and the $2n\times 2n$ covariance matrix $T$ is as in \eqref{eq.T-Hyper}, 
using the hypercontractivity and reverse hypercontractivity inequalities 
\eqref{eq.Hyper} and \eqref{eq.Rev-Hyper}. 

Indeed, suppose for example, that $T\geq P={\rm diag}(qI_n,pI_n)$. Thus $q,p<1$ and 
applying \eqref{eq.DownEquivalence} to $T\geq P$, one sees that $(1-p)/(1-q') \geq e^{-2t}$. 
This allows us to use the pair $p,q'$ in \eqref{eq.Rev-Hyper} and combining this with 
reverse H\"{o}lder's inequality yields
\begin{align*}
 \IE g(X) f(Y) &= \IE g(X){P}_tf(X)
 \geq \|g\|_{L_q(\gamma_n)}\|{P}_tf\|_{L_{q'}(\gamma_n)}
 \geq \|g\|_{L_q(\gamma_n)}\|f\|_{L_{p}(\gamma_n)}.
\end{align*}
This gives Theorem \ref{main} (i). A similar argument, using \eqref{eq.Up-Dual}
instead of \eqref{eq.Down-Dual} also shows $(ii)$ of Theorem \ref{main}.

\smallskip

We note here that the connection between Theorem \ref{main}(i) and the Gaussian 
hypercontractivity is rather classical. We refer to \cite{BL} and to Theorem 13.8.1
in \cite{Garl}.

\section{Theorem 2: A Generalization of the 
         sharp Young  and reverse Young inequalities.} \label{sec5}

\subsection{The sharp Young and reverse Young inequalities}
The sharp Young and reverse Young inequalities states that for nonnegative 
measurable functions $f_1$ and $f_2$ on $\mathbb{R}^n,$ if $p,q,r>0$ satisfy 
$p^{-1}+ q^{-1}= 1 + r^{-1}$, then we have respectively,
\begin{align}\label{sy}
\|f_1*f_2\|_r\leq C^n\|f_1\|_{p}\|f_2\|_{q}\,\,\mbox{for $p,q,r\geq 1$}
\end{align}
and
\begin{align}\label{rsy}
\|f_1*f_2\|_r\geq C^n\|f_1\|_{p}\|f_2\|_{q}\,\,\mbox{for $p,q,r\leq 1$},
\end{align}
where, $C:=C_pC_q/C_r$, where $C_u^2=|u|^{1/u}/|u'|^{1/u'}$ for $1/u+1/u'=1$.

\smallskip

The sharp Young inequality \eqref{sy} was proved by Beckner in \cite{Bec} and shortly 
after, by Brascamp and Lieb in \cite{BL}. In the late paper Brascamp and Lieb proved a
generalization of \eqref{sy}, the so-called Brascamp-Lieb inequality. In addition they
introduced, the reverse inequality \eqref{rsy}. In this section, we prove 
Theorem \ref{Gen-Rev-Young} and we show how it generalizes both sharp Young and reverse 
sharp Young inequalities. 
These fundamental inequalities have many applications in analysis. As it was noticed by 
Brascamp and Lieb in \cite{BL}, from the sharp reverse Young inequality one can retrieve 
the Pr\'{e}kopa-Leindler inequality \cite{Lei,Pre}. On the other hand, Lieb in \cite{Lieb} 
showed that sharp Young inequality implies Shannon's entropy power inequality. Furthermore, 
an argument that the sharp reverse Young inequality interpolates between the Shannon entropy 
power inequality $(r\rightarrow 1-)$ and the Brunn-Minkowski inequality $(r\rightarrow 0+)$, 
is presented in Chapter 17 of the book \cite{CT}.

\smallskip

 Let us first explain how \eqref{sy} and \eqref{rsy} can be recovered, directlly from our 
theorem. Set the matrices
\begin{align*}
 U&=
 \left(\begin{array}{cc}
 I_n & -I_n \\
 \IO &  I_n \\
 I_n &  \IO
 \end{array}\right),\\
\\ 
 B_1&=\left(\begin{array}{cc}
c_{3}(1-c_{3})I_n & (1-c_{2})(1-c_{3})I_n\\
(1-c_{2})(1-c_{3})I_n & c_{2}(1-c_{2})I_n
\end{array}\right),\\
\\
B_2&=\left(\begin{array}{cc}
{{c}}_{3}(1+{{c}}_{3})I_n & (c_{2}-1)(1+c_{3})I_n\\
(c_2-1)(1+c_{3})I_n & c_{2}(c_{2}-1)I_n
\end{array}\right),
\end{align*}
where $c_{1}= p^{-1}, c_{2}= q^{-1}$ and $c_{3}= |r^{\prime}|^{-1}$, and let
\[
P:= {\rm diag}( pI_n, qI_n, r'I_n).
\]
One can check that if $p,q,r\geq 1$, then
$ U  {B_1} U^{\ast} \leq P D_{U {B_1} U^{\ast}}$
and if $0<p,q,r\leq 1$,
$ U B_2U^{\ast} \geq P D_{UB_2U^{\ast}}.
$
In either case, a direct computation gives
\[
\left( \frac{{\rm det} B_i}{({\rm det }U_{1}B_iU_{1}^{\ast})^{\frac{1}{p}} 
({\rm det }U_{2}B_iU_{2}^{\ast})^{\frac{1}{q}} 
({\rm det }U_{3} 
B_i U_{3}^{\ast})^{\frac{1}{r^{\prime}}}}\right)^{\frac{1}{2}} = C^{n} .
\]
Then, for any nonnegative measurable functions $f_1,f_2,g$ on $\mathbb{R}^n$, we
apply Theorem \ref{Gen-Rev-Young}(i) with $U,B_1,P$ and Theorem \ref{Gen-Rev-Young}(ii) 
with $U,B_2,P,$. Finally, the duality relations \eqref{eq.Up-Dual} and \eqref{eq.Down-Dual} 
lead to \eqref{sy} and \eqref{rsy}, respectively. The proceeding argument can be found in
\cite{Le}, see also \cite{BCLM}.

\subsection{Proof of Theorem \ref{Gen-Rev-Young}}
Before we give the proof of Theorem \ref{Gen-Rev-Young}, let us comment on 
its assumptions. Note first that the additional assumption \eqref{eq.homogen} that appears 
in Theorem \ref{Gen-Rev-Young} is actually a necessary condition due to the homogeneity 
of the Lebesgue measure. Moreover, in the following lemma, we shall see that under this 
homogeneity condition, the assumption \eqref{eq.A-R-Y-Up} is equivalent with 
\eqref{eq.App3}.

\begin{lemma}\label{prop1}
In the setting of Theorem \ref{Gen-Rev-Young}, the following are equivalent
 \begin{equation} \label{eq.App1}
\sum_{i=1}^n\frac{n_i}{p_i}=n\;\;\mbox{and}\;\;UBU^* \leq P D_{UBU^*},
 \end{equation}
 \begin{equation}\label{eq.App2}
  B^{-1} = U^{\ast} (P D_{UBU^*})^{-1} U.
 \end{equation}
 \begin{equation} \label{eq.App3}
  B^{-1} = \sum_{i=1}^{m} \frac{1}{p_{i}} U_i^{\ast}
  \left(U_i B U_i^{\ast}\right)^{-1} U_i
 \end{equation}
\end{lemma}
\begin{proof}  
 Clearly \eqref{eq.App2} and \eqref{eq.App3} are equivalent.
 Let's see first how \eqref{eq.App2} implies \eqref{eq.App1}. 
 Write $B=\Sigma\Sigma^*$ and $C_B^{-1}:=PD_{UBU^*}$. Then \eqref{eq.App2} 
 can be written equivalently as
 \begin{equation} \label{eq.lehec}
  \big(U\Sigma\big)^*C_B\big(U\Sigma\big) 
  =\left(\sqrt{C_B}U\Sigma\right)^*
   \left(\sqrt{C_B}U\Sigma\right) = I_n
 \end{equation}
 and so $\left(\sqrt{C_B}U\Sigma\right) \left(\sqrt{C_B}U\Sigma\right)^* \leq I_N$
 or equivalently $UBU^* \leq P D_{UBU^*}$.
 The homogeneity condition $\sum_{i=1}^mn_i/p_i=n$ follows by taking trace in 
 \eqref{eq.lehec}. Indeed, if we set $U_{i\Sigma}:=U_i\Sigma$, a direct computation 
 shows that
 \begin{eqnarray*}
 (U\Sigma)^*\,C_B\,(U\Sigma) &=& 
     \sum_{i=1}^m c_i\,(U_i\Sigma)^*\,
     \big(U_i B U_i^*\big)^{-1}\,(U_i\Sigma) \\
 &=& \sum_{i=1}^m c_i\;U_{i\Sigma}^*\,\big(U_{i\Sigma}
     \,U_{i\Sigma}^*\big)^{-1}\, U_{i\Sigma} 
 = \sum_{i=1}^m c_i\;\big({\widetilde U_{i\Sigma}}\big)^*\;{\widetilde U_{i\Sigma}}
\end{eqnarray*}
where, ${\widetilde U_{i\Sigma}}:=\big(U_{i\Sigma}\,U_{i\Sigma}^*\big)^{-1/2}\,U_{i\Sigma}$. 
Note that ${\widetilde U_{i\Sigma}}\,\big({\widetilde U_{i\Sigma}}\big)^* =I_{n_i}$, and thus,
\begin{equation}\label{trace}
 n={\rm tr}(I_n)=
 {\rm tr}\Big(\big(U\Sigma\big)^*\,C_B\,\big(U\Sigma\big)\Big) =
 {\rm tr}\Big(\sum_{i=1}^m c_i
 \;\big({\widetilde U_{i\Sigma}}\big)^*\;{\widetilde U_{i\Sigma}}\Big)
 = \sum_{i=1}^m c_in_i=\sum_{i=1}^m {n_i}{p_i}.
\end{equation}
To see why \eqref{eq.App1} implies \eqref{eq.App2} recall that $UBU^* \leq P D_{UBU^*}$
can be written equivalently as 
$\left(\sqrt{C_B}U\Sigma\right)\left(\sqrt{C_B}U\Sigma\right)^* \leq I_N$ which implies
that 
\begin{equation}\label{eq.leq}
 \big(U\Sigma\big)^*C_B\big(U\Sigma\big) \leq I_n
\end{equation}
To complete the proof we have to show that equality holds in \eqref{eq.leq}. Indeed,
note that if $A_{1}, A_{2}$ are two positive definite matrices with $A_{1} \leq A_{2}$
and ${\rm tr} (A_{1} )= {\rm tr}(A_{2})$ then $A_{1}= A_{2}$. Thus, under the homogeneity 
condition, $\sum_{i=1}^mn_i/p_i=n$ we get that
\[
 {\rm tr}\Big(\big(U\Sigma\big)^*\,C_B\,\big(U\Sigma\big)\Big) 
 = \sum_{i=1}^m {n_i}{p_i} = n ={\rm tr}(I_n),
\]
and so equality holds in \eqref{eq.leq}. 
\end{proof}

\begin{remark}\rm
In \cite{Le}, Lehec proved a reformulation of the Brascamp-Lieb 
inequality, which states that \eqref{eq.R-Y-Up-intro} holds true under the 
assumption \eqref{eq.App3}. As an immediate consequence of Lemma \ref{prop1}, 
Theorem \ref{Gen-Rev-Young}(i) is exactly the Brascamp-Lieb inequality. We refer
to \cite{Le} for more details.
\end{remark}

\smallskip

We close this section with the proof of Theorem \ref{Gen-Rev-Young}.

\begin{proof}[Proof of Theorem \ref{Gen-Rev-Young}]
We prove (i) first. Note that under the given assumptions, we have that
 ${\rm rank}({ U}_i)=n_i\leq n={\rm rank}(B)$, for every $i\leq m$. 
 Thus, the
 $n_i\times n_i$ matrix $B_i:={ U}_iB{ U}_i^*$ has full rank $n_i$. 
 Consider a Gaussian random vector
 $$
 X=(X_i,\ldots,X_m)\sim N({ 0},UBU^*),
 $$
 where $X_i \sim N({ 0}, B_i)$. Using assumption \eqref{eq.A-R-Y-Up}, 
 we apply Theorem \ref{main}(i) to get that
 \begin{equation}\label{eq.Thrm1ii}
  \IE \prod_{i=1}^m f_i(X_i)\leq
  \prod_{i=1}^m \Big(\IE f_i(X_i)^{p_i}\Big)^{1/p_i}.
 \end{equation}
 Write $B=\Sigma\Sigma^*$ for some nonsingular matrix $\Sigma$. 
 Then the covariance matrix of $X_i$ can be written as
 $
 B_i={U}_iB{ U}_i^*=({U}_i\Sigma)({U}_i\Sigma)^*.
 $
 Thus, by the change of variables $y=\Sigma x$, we have that
 \begin{align*}
  \IE \prod_{i=1}^m f_i(X_i) &=
  \int_{\IR^n} \prod_{i=1}^m f_i \big({ U}_i \Sigma x \big)d\gamma_n(x) \\
  & =  \frac{1}{(2 \pi)^{\frac{n}{2}}\det(B)^{\frac{1}{2}}}
        \int_{\IR^n}\prod_{i=1}^m f_i\big({U}_i x\big)\exp\left(-\frac{1}{2} 
        \inner{x}{B^{-1}x}\right)dx.
 \end{align*}
 On the other hand,
 $$
 \IE f_i(X_i)^{p_i}= \frac{1}{(2 \pi)^{\frac{n_i}{2}}\det(B_i)^{\frac{1}{2}}}
        \int_{\IR^{n_i}}f_i(x_i)^{p_i}
        \exp\left(-\frac{1}{2} \inner{x_i}{B_i^{-1}x_i}\right)dx_i.
 $$
Finally, taking $\sigma B$ instead of $B$ for $\sigma>0$ and using the homogeneity 
condition \eqref{eq.homogen}, we can write \eqref{eq.Thrm1ii} equivalently as
 \begin{align*}
   &\int_{\IR^n} \prod_{i=1}^m f_i\big({ U}_i x\big)
    \exp\left(-\frac{1}{2\sigma} \inner{x}{B^{-1}x}\right)dx\\
   &\leq \left(\frac{\det(B)}{\prod_{i=1}^m 
    \det(B_i)^{\frac{1}{p_i}}}\right)^{\frac{1}{2}}
    \prod_{i=1}^m \left(\int_{\IR^{n_i}} f_i(x_i)^{p_i}
    \exp\left(-\frac{1}{2\sigma}
    \inner{x_i}{B_i^{-1}x_i}\right)dx_i\right)^{\frac{1}{p_i}}.
 \end{align*}
Letting $\sigma\rightarrow+\infty$, we get \eqref{eq.R-Y-Up-intro}. 

For the equality case, using lemma \ref{prop1}, and taking the functions 
\[
 f_i(x_i)=\exp(-p_i^{-1}\left<B_i^{-1}x_i,x_i\right>), 
\]
a direct computation  gives the equality in \eqref{eq.R-Y-Up-intro}. 
To prove (ii), one may proceed similarly by using Theorem \ref{main}(ii).
\end{proof}

\section{
         A generalization of Barthe's lemma. }\label{sec6}

We begin this section by recalling the following result of F. Barthe from \cite{BAYoung}.

\begin{proposition}
\label{Barthe's Lemma}
Assume that $p,q,r>1$ with $1/p+ 1/q = 1+ 1/r$ and set $c=\sqrt\frac{r^\prime}{q^\prime}$ 
and $s=\sqrt\frac{r^\prime}{p^\prime}$. For any $f,g,F, G$ continuous and positive 
functions in $L^{1}(\mathbb R)$ satisfying $\int f = \int F$ and $\int g = \int G$, we 
have that
\begin{align}
\begin{split}\label{BartheLemma-1}
&\left( \int \left( \int f^{\frac{1}{p}} (c x -s y ) \,
 g^{\frac{1}{q}} ( s x + c y ) \,dx \right)^{r} d y \right)^{\frac{1}{r}} \\
& \qquad \qquad \leq \int \left( \int F^{\frac{r}{p}} (c X - s Y) \,
  G^{\frac{r}{q}} (s X + c Y) \,dY \right)^{\frac{1}{r}} dX .
\end{split}
\end{align}
\end{proposition}

As we have mentioned in the introduction, starting from this lemma, Barthe 
presented a simplified proof for both the sharp Young and reverse Young 
inequalities (\cite{BAYoung}). Later he further showed in \cite{BaArxiv}, 
that a generalization of this lemma to more than two functions, can been 
used to prove the rank 1 case of both, the Brascamp-Lieb and Barthe inequalities. 
In this section, using Theorem \ref{Gen-Rev-Young}, we will derive a more general 
form of his lemma, from wich one can retrieve the general case (${\rm rank} > 1$) 
of the Brascamp-Lieb and Barthe inequalities. 

For notational convenience, we set the following two conditions, that we will use 
throughout this section. 
\begin{itemize}
\item[$(A1)$] Let $m,n,n_{1},\ldots,n_{m}$ be positive integers. Denote by ${U}_{i}$ 
              a $n_{i}\times n$ matrix with $\mbox{rank}({U}_i)=n_i$ for $i\leq m$.
              Set $N=\sum_{i=1}^mn_i$, and let ${U}$ be the $N\times n$ matrix with 
              block rows ${U}_1,\ldots,{U}_m$.
\item[$(A2)$] Let $c_1,\ldots,c_m$ be positive numbers and $A$ be a 
              $n\times n$-dimensional real symmetric and positive definite matrix. 
              Set $A_i={U}_iA{U}_i^*$ for $i\leq m$ and suppose that
              \begin{align}\label{sec6.1:add:eq1}
               {U}^*C_A{U}=A^{-1},
              \end{align}
              where $C_{A}:={\rm diag}\bigl(c_{1} A_{1}^{-1},\ldots,c_{m} A_{m}^{-1}\bigr)$.
\end{itemize}

\begin{remark}\rm 
 Since $\mbox{rank}(U_i)=n_i$ and ${A}$ is symmetric positive definite, 
 it implies that $n_i\leq n$, $A_i^{-1}$ exists and $C_A$ is well-defined. Moreover, 
 from \eqref{sec6.1:add:eq1}, Lemma \ref{prop1} implies the homogeneity condition  
 \begin{equation*} \label{Leh-1}
  \sum_{i=1}^{m} c_{i} n_{i} = n
 \end{equation*}
 together with $UAU^*\leq C_A^{-1}$. Thus, $N\geq n$.
\end{remark}

\begin{remark}\rm 
 The assumption \eqref{sec6.1:add:eq1} ensures that there exists a $N\times(N-n)$ 
 matrix $W$ with $\mbox{rank}(W)=N-n$ such that 
 \begin{equation} \label{THEprop-Leh-2}
  \sqrt{C_{A}} {{U}} A {{U}}^{\ast} \sqrt{C_{A}} + W W^{\ast} = I_{N}.
 \end{equation}
 In other words, the row vectors of the $N\times N$ matrix $(\sqrt{C_A}{U}\sqrt{A},W)$ 
 form an orthonormal basis of $\mathbb{R}^N.$ To see this, recall that for any real 
 matrix $M$, $M^*M$ and $MM^*$ have the same non-zero eigenvalues with the same algebraic 
 multiplicities. If we set $M= \sqrt{C_A}U\sqrt{A}$, then \eqref{sec6.1:add:eq1} reads 
 $M^*M=I_n$ and thus, $I_N-\sqrt{C_{A}} {{U}} A {{U}}^{\ast} \sqrt{C_{A}}=I_N-MM^*$ is 
 symmetric positive semi-definite and has rank $N-n$. Lemma \ref{lemma.Umatrix} guarantees 
 then, the existence of $W$.
\end{remark}

\begin{theorem}\label{THE prop}
 Assume that $(A1)$ and $(A2)$ hold and $W$ satisfies \eqref{THEprop-Leh-2}. 
 For $\rho>0$, set
 \begin{align}\label{add:eq3}
  \Gamma_\rho&=\frac{(\det(A))^{\frac{1}{2}}}{\rho^{\frac{N-n}{2\rho}}}
  \prod_{i=1}^m\frac{(c_ip_i)^{\frac{n_i}{2p_i}}}{(\det(A_i))^{\frac{1}{2p_{i}}}},
 \end{align}
 where
 \begin{equation} \label{THEprop-3}
  p_{i}:=\frac{1}{c_{i} ( 1+ \frac{1-c_{i}}{\rho c_{i}})}.
 \end{equation}
 Let $f_i$ be nonnegative measurable function on $\mathbb{R}^{n_i}$ for $i\leq m$ 
 and define 
 \[
 F(x,y)=\prod_{i=1}^m f_i\biggl({U}_ix+\frac{1}{\sqrt{c_i}}
 \sqrt{A_i}W_iy\biggr),\quad (x,y)\in\mathbb{R}^{n}\times\mathbb{R}^{N-n}.
 \]
 $(i)$ For $\rho\geq 1$,
  \begin{equation}\label{THEprop-Leh-4}
   \int_{\mathbb R^{n}} \left(\int_{\IR^{N-n}} F^\rho(x,y)\,dy\right)^{\frac{1}{\rho}}dx
   \geq \Gamma_\rho\prod_{i=1}^m\|f_{i}\|_{p_{i}} \geq 
   \left( \int_{\mathbb R^{N-n}} 
   \left( \int_{\mathbb R^{n}}F(x,y)\,dx \right)^{\rho} d y \right)^{\frac{1}{\rho}}.
  \end{equation}
 $(ii)$ For $0<\rho\leq 1$,
  \begin{equation}\label{THEprop-Leh-5}
   \int_{\mathbb R^{n}} \left(\int_{\IR^{N-n}} F^\rho(x,y)\,dy\right)^{\frac{1}{\rho}}dx
   \leq \Gamma_\rho\prod_{i=1}^m\|f_{i}\|_{p_{i}} \leq 
   \left( \int_{\mathbb R^{N-n}} 
   \left( \int_{\mathbb R^{n}}F(x,y)\,dx \right)^{\rho} d y \right)^{\frac{1}{\rho}}.
  \end{equation}
 Moreover, one has equality in \eqref{THEprop-Leh-4} and \eqref{THEprop-Leh-5} if 
 $\rho=1$ ( for any choice of the $f_i$'s) or if 
 \[
 f_{i}(x)=\exp\bigl(-c_{i}\left<x,A_i^{-1}x\right>\bigr),
 \;\;x\in\IR^{n_i} \qquad 1\leq i\leq m
 \]
 (for any choice of $\rho >0$).
\end{theorem}


The verification of the equality cases in \eqref{THEprop-Leh-4} and \eqref{THEprop-Leh-5} 
can be easily carried out with a routine computation and will be omitted. To show 
\eqref{THEprop-Leh-4} and \eqref{THEprop-Leh-5}, we will apply Theorem \ref{Gen-Rev-Young} 
with suitable chosen matrices. The proof consists in two parts. In the first part 
we consider the ``geometric" form, where ${A=I_n}$ and ${{U}_i{U}_i^*=I_{n_i}}$ for 
${i\leq m}$, while in the second part we deal with the general form of the theorem. 

The core of our argument is based on a idea of Brascamp and Lieb from \cite{BL}, where
the authors proved that, the Precopa-Leindler inequality can be retrieved from their
reverse sharp Young inequality. Theorem \ref{THE prop} is the result of our effort to 
generalize this proof of Brascamp and Lieb, in order to retrieve Barthe's inequality, 
from our Theorem \ref{Gen-Rev-Young}(ii).

First, we recall a standard result for positive definite matrices, for the proof of 
which we refer to Theorem 1.3.3 in the book \cite{Bha} 

\begin{lemma}\label{Lemma-Bha}
Let $k,d$ be positive integers, $A$ and $B$ be two $k\times k$ and $d\times d$ real 
symmetric and positive definite matrices respectively and $X$ be a $d\times k$ matrix. 
Then
\begin{equation}
\left(\begin{array}{cc}
A & X^* \\
X & B
\end{array}\right) \geq 0
\;\Leftrightarrow\;
B - XA^{-1}X^* \geq 0.
 \end{equation}
\end{lemma}

\begin{proof}[Proof of Theorem \ref{THE prop}]
 We recall some standard notations. For any positive integers $k,d,$ we denote the 
 $k\times d$-dimensional zero matrix by $\mathbb{O}_{k\times d}$ or simply $\mathbb{O}$ 
 whenever there is no ambiguity. For any real number $r,$ $r'$ stands for the H\"{o}lder 
 conjugate exponent. Since for $\rho=1$, Theorem \ref{THE prop} holds trivially true by 
 Fubini's Theorem, we may assume without loss of generality that $\rho\neq 1$.
 The first part of our argument runs as follows.

 \smallskip

 \noindent {\bf Part I} : ${A=I_n}$ and ${{U}_i{U}_i^*=I_{n_i}}$ for ${i\leq m}$.

 In this case, we can rewrite  
 \eqref{sec6.1:add:eq1} and \eqref{THEprop-Leh-2} as
 \begin{align}
 \begin{split}
 \label{add:sc:eq2}
 {U}^*C{U}&=I_n,
 \end{split}\\
 \begin{split}
 \label{add:sc:eq3}
 \sqrt{C}{U}{U}^*\sqrt{C}+WW^*&=I_N,
 \end{split}
 \end{align}
 where $C:=C_{I_n}={\rm diag}(c_1I_{n_1},\ldots,c_mI_{n_m})$. 
 Set $g_i$ by $g_{i}(\sqrt{c_i}x_i)=f_i(x_i)$ for $x_i\in\mathbb{R}^{n_i}$ and define
 \begin{align*}
  {{G}}(x)&= \int_{\IR^{N-n}} \prod_{i=1}^m
  g_{i}^{\rho}  \big( \sqrt{c_i}{U}_{i}x + W_{i}y\big)dy,\qquad x\in\IR^n,
 \end{align*}
 and
 \begin{align*}
  \tilde{{G}}(y)&= \int_{\IR^{n}} \prod_{i=1}^m
   g_{i}\big( \sqrt{c_i}{U}_{i}x + W_{i}y\big)\,dx,\qquad y\in\IR^{N-n}.
 \end{align*}
 First, we prove the left-hand side of \eqref{THEprop-Leh-4}. 
 We set $r:=1/\rho$, and using the duality relation \eqref{eq.Down-Dual}
 we write
 \begin{align*}
  \int_{\mathbb R^{n}} \left(\int_{\IR^{N-n}} F^\rho(x,y)\,dy\right)^{\frac{1}{\rho}}\,dx
   &=\int_{\mathbb R^n} 
     \left(\int_{\IR^{N-n}} \prod_{i=1}^m
     g_{i}^{\rho}  \big( \sqrt{c_i}U_ix + W_iy\big)\,dy\right)^{\frac{1}{\rho}}\,dx \\
   &=\norm{G}_r^r \\
   &=\left( \inf_{\norm{H}_{r^\prime}=1}
     \int_{\mathbb R^n} H(x)G(x)\,dx\right)^r \\
   &=\left( \inf_{\norm{H}_{r^\prime}=1}
     \int_{\mathbb R^N} H(V_0 z) \prod_{i=1}^m
     g_{i}^\rho\big( V_i z\big)\,dz\right)^{1/\rho}
 \end{align*}
 where $V_0=\bigl(I_n\;\,\mathbb{O}_{n\times(N-n)}\bigr)$, and 
 $V_i=\bigl(\sqrt{c_i}U_i\;\,W_i\bigr)$, $1\leq i\leq m$.
 
 In what follows, we prove that the following inequality holds true
 \begin{align} \label{eq.ProofTh5.01}
  \int_{\mathbb R^N} H(V_0 z) \prod_{i=1}^m g_i^\rho\big( V_i z\big)\,dz
  \,\geq\, \Gamma_\rho^\rho \,\prod_{i=1}^m\|f_i\|_{p_i}^\rho,
 \end{align}
 for every nonnegative $H$ such that $\norm{H}_{r^\prime}=1$, and this will 
 prove the left hand side of \eqref{THEprop-Leh-4}. To do so, we use Theorem 
 \ref{Gen-Rev-Young}(ii) for the following choice of matrices:
 \begin{align}
  V = \left(\begin{array}{c}
  V_0\\
  V_1\\
  \vdots\\
  V_m \end{array}\right)
  = \left(\begin{array}{cc}
     I_n    & {\mathbb O}_{n\times (N-n)} \\
  \sqrt{C}U &           W                 \\
  \end{array}\right)
  \end{align}
  and
  \begin{align}
  B=\left(\begin{array}{cc}
  I_{n} & {\mathbb O} \\
  \mathbb O & \frac{1}{\rho} I_{N-n}
  \end{array}\right),\;\; 
  Q={\rm diag}\big(r'I_n,q_1 I_{n_{1}},\ldots,q_m I_{n_{m}}\big),
 \end{align}
 where $q_{i}:= p_i/\rho$. 
 
 \smallskip
 
 A straightforward computation gives that
 \begin{equation*}
  VBV*=\left(\begin{array}{cc}
     I_n    & U^*\sqrt{C} \\
  \sqrt{C}U &    R_\rho   \\
  \end{array}\right)
 \end{equation*}
 where $R_\rho=\sqrt{C}UU^*\sqrt{C} +\frac1\rho WW^*$. So, using the identity 
 \eqref{add:sc:eq3} and the assumption that $U_iU_i^*=I_{n_i}$, we get
 \begin{align*}
  D_{VBV^*}&:={\rm diag}\Big({V}_0B{V}_0^*,{V}_{1} 
      B{V}_{1}^{\ast},\ldots, {V}_{m}B{V}_{m}^{\ast}\Big)\nonumber \\
   & ={\rm diag}\biggl(I_{n},
      \Bigl(c_1+\frac{1-c_1}{\rho}\Bigr)I_{n_{1}},
      \ldots,\Bigl(c_m+\frac{1-c_m}{\rho}\Bigr)I_{n_{m}}\biggr) \nonumber\\
   & ={\rm diag}\biggl(I_{n},
      \frac{1}{p_1}I_{n_{1}},\ldots,\frac{1}{p_m}I_{n_{m}}\biggr)
 \end{align*}

 In order to apply Theorem \ref{Gen-Rev-Young}(ii), for this set of matrices, we 
 need to check its assumptions. Recall first that $\sum_{i=1}^{m} c_{i} n_{i}= n$ 
 and $\sum_{i=1}^{m} n_{i} = N$, and so 
 \[ 
 \frac{n}{r^\prime} + \sum_{i=1}^m \frac{n_i}{q_i}
 = n(1-\rho) + \rho \sum_{i=1}^m n_ic_i + \sum_{i=1}^m n_i (1-c_i) = N,
 \]
 and thus the homogeneity condition \eqref{eq.homogen} holds true. Moreover we 
 need to check that
 \begin{align}\label{add:eq1}
  VBV^*-QD_{VBV^*}=
  \left(\begin{array}{cc}
   (1-r^\prime) I_{n} & {U}^{\ast} \sqrt{C} \\
   \sqrt{C} {U} & \Delta_\rho
 \end{array}\right) \geq 0,
 \end{align}
 where 
 \begin{align*}
  \Delta_\rho 
   &= \sqrt{C}UU^*\sqrt{C}+\frac{1}{\rho}\,WW^* - 
      {\rm diag}\biggl(\frac{q_1}{p_1}I_{n_{1}},\ldots,\frac{q_m}{p_m}I_{n_{m}}\biggr) \\
   &= \sqrt{C}UU^*\sqrt{C} + WW^* - \left(1-\frac{1}{\rho}\right)WW^* - \frac1\rho\,I_N \\
   &= \left(1-\frac{1}{\rho}\right)\bigl(I_N-WW^*\bigr) 
    = \left(1-\frac{1}{\rho}\right) \sqrt{C}UU^*\sqrt{C}
 \end{align*}
 Note finally, that using the identity \eqref{add:sc:eq3} again, we get
 \begin{align*}
  \Delta_\rho - \sqrt{C}U \Bigl((1-r^\prime)\,I_n\Bigr)^{-1} U^*\sqrt{C} =
  \Delta_\rho - \left(1-\frac{1}{\rho}\right)\sqrt{C}UU^*\sqrt{C} = 0.
 \end{align*}
 and so by lemma \ref{Lemma-Bha} we have that \eqref{add:eq1} holds also true. 

 So by applying Theorem \ref{Gen-Rev-Young}(ii), we get that 
 \begin{align*} 
  \int_{\mathbb R^N} H(V_0 z) \prod_{i=1}^m
      g_{i}^\rho\big( V_i z\big)\,dz 
  &\geq \left(\frac{\det(B)}{\det(V_0BV_0^*)^{\frac{1}{r'}}
        \prod_{i=1}^m\det(V_iBV_i^*)^{\frac{1}{q_i}}}\right)^{\frac{1}{2}}
        \prod_{i=1}^m\|g_i^\rho\|_{q_i}\\
  &= \left(\frac{\rho^{-{(N-n)}}}{1^{\frac{1}{r'}}
     \prod_{i=1}^mp_i^{-\frac{\rho n_i}{p_i}}}\right)^{\frac{1}{2}}
     \prod_{i=1}^mc_i^{\frac{\rho n_i}{2p_i}}\|f_i\|_{p_i}^\rho\\
  &= \Gamma_\rho^\rho\prod_{i=1}^m\|f_i\|_{p_i}^\rho,
 \end{align*}
 where in the last equality we used that $\sum_{i=1}^mc_in_i=n$, $\sum_{i=1}^mn_i=N$ 
 and the change of variables $\|g_i\|_{p_i/\rho}=c_i^{\rho n_i/2p_i}\|f_i\|_{p_i}^\rho$. 
 This proves \eqref{eq.ProofTh5.01}, and thus the left-hand side of \eqref{THEprop-Leh-4}.

 \smallskip

 We now turn to the right-hand side of \eqref{THEprop-Leh-4}. Using the duality relation 
 \eqref{eq.Up-Dual} we have that
 \begin{align*}
  \left( \int_{\mathbb R^{N-n}} 
  \left( \int_{\mathbb R^{n}}F(x,y) \,dx \right)^{\rho} d y \right)^{\frac{1}{\rho}}
   &= \left(\int_{\mathbb R^{N-n}} 
      \left(\int_{\IR^n} \prod_{i=1}^m
      g_i\big(\sqrt{c_i}U_ix + W_iy\big)\,dx\right)^\rho dy\right)^{\frac{1}{\rho}} \\
   &= \norm{\tilde G}_\rho \\
   &= \sup_{\norm{H}_{\rho^\prime}=1}
      \int_{\mathbb R^{N-n}} H(y){\tilde G}(y)\,dx \\
   &= \sup_{\norm{H}_{\rho^\prime}=1}
      \int_{\mathbb R^N} H({\tilde V}_0 z) \prod_{i=1}^m
      g_{i}\big( V_i z\big)\,dz
 \end{align*}
 where ${\tilde V}_0=\bigl(\mathbb{O}_{(N-n)\times n}\;\,I_{N-n}\bigr)$, and 
 $V_i=\bigl(\sqrt{c_i}U_i\;\,W_i\bigr)$, $1\leq i\leq m$, as before. Similarly 
 to the previous case, we will show that the following inequality holds true
 \begin{align} \label{eq.ProofTh5.02}
  \int_{\mathbb R^N} H({\tilde V}_0 z) \prod_{i=1}^m g_{i}\big( V_i z\big)\,dz
  \,\leq\, \Gamma_\rho \,\prod_{i=1}^m\|f_i\|_{p_i},
 \end{align}
 for every $H$ such that $\norm{H}_{\rho^\prime}=1$, and this will prove the 
 right-hand side of \eqref{THEprop-Leh-4}. Now we will use Theorem 
 \ref{Gen-Rev-Young}(i) for the matrices:
 \begin{align}
  {\tilde V} =
  \left(\begin{array}{c}
  {\tilde V}_0\\
  V_1\\
  \vdots\\
  V_m
  \end{array}\right)
  = \left(\begin{array}{cc}
  {\mathbb O}_{(N-n)\times n}  &  I_{N-n} \\
            \sqrt{C}U          &    W     \\
  \end{array}\right)
  \end{align}
  and
  \begin{align}
  {\tilde B} 
  =\left(\begin{array}{cc}
  \rho\,I_{n} & {\mathbb O} \\
   \mathbb O  &  I_{N-n}
  \end{array}\right),\;\; 
  P={\rm diag}\big(\rho^\prime\,I_{N-n}\,,\,
  p_1 I_{n_{1}}\,,\,\ldots\,,\,p_m I_{n_{m}}\big).
 \end{align}
 
 \smallskip
 
 This time we have that
 \begin{equation*}
  {\tilde V}{\tilde B}{\tilde V}^*
  =\left(\begin{array}{cc}
    I_{N-n} & W^* \\
       W    &    {\tilde R}_\rho   \\
  \end{array}\right)
 \end{equation*}
 where ${\tilde R}_\rho=\rho \sqrt{C}UU^*\sqrt{C} + WW^*$. Using the identity 
 \eqref{add:sc:eq3} and the fact $U_iU_i^*=I_{n_i}$, we get
 \begin{align*}
  D_{{\tilde V}{\tilde B}{\tilde V}^*}
   &:={\rm diag}\Big({\tilde V}_0{\tilde B}{\tilde V}_0^*,
     {\tilde V}_1{\tilde B}{\tilde V}_1^*,\ldots, 
     {\tilde V}_m{\tilde B}{\tilde V}_m^*\Big)\nonumber \\
   & ={\rm diag}\biggl(I_{N-n},
      \rho\Bigl(c_1+\frac{1-c_1}{\rho}\Bigr)I_{n_{1}},
      \ldots,\rho\Bigl(c_m+\frac{1-c_m}{\rho}\Bigr)I_{n_{m}}\biggr) \nonumber\\
   & ={\rm diag}\biggl(I_{N-n},
      \frac{\rho}{p_1}I_{n_{1}},\ldots,\frac{\rho}{p_m}I_{n_{m}}\biggr)
 \end{align*}

 To apply Theorem \ref{Gen-Rev-Young}(i), for this set of matrices, we check its 
 assumptions. Note first that
 \[ 
 \frac{N-n}{\rho^\prime} + \sum_{i=1}^m \frac{n_i}{p_i}
 = (N-n)\frac{\rho-1}{\rho} + 
 \sum_{i=1}^m n_ic_i + \sum_{i=1}^m n_i\frac{1-c_i}{\rho} = N,
 \]
 and thus the homogeneity condition \eqref{eq.homogen} holds true. We need also to 
 check that
 \begin{align}\label{add:eq2}
  {\tilde V}{\tilde B}{\tilde V}^*-PD_{{\tilde V}{\tilde B}{\tilde V}^*}
  =\left(\begin{array}{cc}
   (1-\rho^\prime) I_{N-n} & W^* \\
   W    &   {\tilde \Delta}_\rho
  \end{array}\right) \leq 0,
 \end{align}
 where 
 \begin{align*}
 {\tilde\Delta}_\rho 
  &= \rho\,\sqrt{C}UU^*\sqrt{C} + WW^* - 
     {\rm diag}\biggl(\rho\,I_{n_{1}},\ldots,\rho\,I_{n_{m}}\biggr) \\
  &= \sqrt{C}UU^*\sqrt{C} + 
     WW^* - \left(1-\rho\right)\sqrt{C}UU^*\sqrt{C} - \rho\,I_N \\
  &= \left(1-\rho\right)\bigl(I_N-\sqrt{C}UU^*\sqrt{C}\bigr)
   = \left(1-\rho\right)\,WW^*.
 \end{align*}
 Note finally, that using the identity \eqref{add:sc:eq3} again, we get
 \begin{align*}
  {\tilde\Delta}_\rho - W \Bigl((1-\rho^\prime)\,I_n\Bigr)^{-1} W^* =
  {\tilde\Delta}_\rho - (1-\rho)\,WW^*= 0.
 \end{align*}
 and so by lemma \ref{Lemma-Bha} we have that \eqref{add:eq2} holds also true. 
 Applying now Theorem \ref{Gen-Rev-Young}(i), we get that 
 \begin{align*}
  \int_{\mathbb R^N} H({\tilde V}_0 z) \prod_{i=1}^m g_{i}\big( V_i z\big)\,dz
  &\leq \left(\frac{\det({\tilde B})}{\det({\tilde V}_0
        {\tilde B}\,{\tilde V}_0^*)^{\frac{1}{\rho'}}
        \prod_{i=1}^m\det(V_i{\tilde B}\,V_i^*)^{\frac{1}{p_i}}}
        \right)^{\frac{1}{2}}\prod_{i=1}^m\|g_i\|_{{p_i}}\\
  &= \left(\frac{\rho^{n}}{\prod_{i=1}^m 
     \left(\frac{\rho}{p_i}\right)^{\frac{n_i}{p_i}}}\right)^{\frac{1}{2}}
     \prod_{i=1}^mc_i^{\frac{ n_i}{2p_i}}\|f_i\|_{p_i}\\
  &= \Gamma_\rho\prod_{i=1}^m\|f_i\|_{p_i},
 \end{align*}
 where in the last equality we used that $\sum_{i=1}^mc_in_i=n$, 
 $\sum_{i=1}^mn_i=N$ and the change of variables 
 $\|g_i\|_{p_i}=c_i^{n_i/2p_i}\|f_i\|_{p_i}$. This proves \eqref{eq.ProofTh5.02}, 
 and thus the right-hand side of \eqref{THEprop-Leh-4}.

 \smallskip

 The proof of \eqref{THEprop-Leh-5} for $0<\rho\leq 1$ is identical and is omitted.

\smallskip

\noindent {\bf Part II} : The general case.

 Now we are not assuming that $A=I_n$ and $U_iU_i^*=I_{n_i}$. In order to reduce the 
 general case to the previous case, we set ${\widetilde U}_i=\sqrt{A_i} U_i \sqrt{A}$ 
 and let ${\widetilde U}$ be the $N\times n$ matrix with block rows 
 ${\widetilde U}_1,\ldots,{\widetilde U}_m$. Then, by the assumptions $(A1)$ and 
 $(A2)$, one sees that 
 \begin{align*}
  {\widetilde{{U}}}_{i}{\widetilde{{U}}}_{i}^*&=I_{n_i},\,\,\forall i\leq m,\\
  \widetilde{U}^*C\widetilde{U}&=I_n,\\
  \sqrt{C}{\widetilde{{U}}} {\widetilde{{U}}}^{\ast} \sqrt{C} + WW^{\ast} &= I_{N},
 \end{align*}
 where $C:={\rm diag}(c_1I_{n_1},\ldots,c_nI_{n_m})$. 
 We define $h_{i}(x):= f_i( \sqrt{A_i} x)$, $x\in\IR^{n_i}$, and using the
first case, we can apply the left-hand side of \eqref{THEprop-Leh-4}, to 
 $h_i's$ and $\widetilde{U}$, and get that if for example $\rho\geq 1$,
 \begin{align} \label{THEprop-Leh-pr-3}
  \int_{\mathbb R^{n}} \left(\int_{\IR^{N-n}} \prod_{i=1}^m
  h_{i}^{\rho}\biggl( {\widetilde{{U}}}_{i} x +
  \frac{1}{\sqrt{c_{i}}} W_{i} y \biggr)dy\right)^{\frac{1}{\rho}}dx  
  \geq \frac{1}{ \rho^{\frac{N-n}{2\rho}}} \prod_{i=1}^{m}
  (c_{i} p_{i})^{\frac{n_{i}}{2p_{i}}}\|h_{i}\|_{p_{i}}.
 \end{align}
 The change of variables $x\mapsto \sqrt{A}x$ leads to
 \begin{align} \label{THEprop-Leh-pr-4}
    \int_{\mathbb R^{n}} \left(\int_{\IR^{N-n}} \prod_{i=1}^m
     h_{i}^{\rho}\biggl( {\widetilde{{U}}}_{i} x +
     \frac{1}{\sqrt{c_{i}}} W_{i} y \biggr)dy\right)^{\frac{1}{\rho}}dx 
    =\frac{1}{ {\rm det}(A)^{\frac{1}{2}}} \int_{\mathbb R^{n}} 
     \left(\int_{\IR^{N-n}} F^\rho \,dy\right)^{\frac{1}{\rho}}d{x}
 \end{align}
 and again, $x_i\mapsto \sqrt{A_i}\,x_i$ gives
 \begin{equation} \label{THEprop-Leh-pr-5}
  \|h_{i}\|_{p_{i}} = 
  \left(\int_{\mathbb R^{n_i}}f_i^{p_i}
  \bigl(\sqrt{A_i}\,x_i\bigr)\,dx_i\right)^{\frac1{p_i}} 
  = \frac1{{\rm det}(A_i)^{\frac1{2p_i}}} \|f_i \|_{p_i} .
 \end{equation}
 Combining \eqref{THEprop-Leh-pr-3}, \eqref{THEprop-Leh-pr-4} and 
 \eqref{THEprop-Leh-pr-5} implies 
 \begin{align*} \label{THEprop-Leh-pr-6}
   \int_{\mathbb R^{n}} \left(\int_{\IR^{N-n}} F^\rho(x,y)\,dy\right)^{\frac{1}{\rho}}d{x}  
   \geq \frac{{\rm det}( A)^{\frac{1}{2}}}{ \rho^{\frac{N-n}{2\rho}}} \prod_{i=1}^{m}
   \left(\frac{(c_i p_i)^{{n_i}}}{ {\rm det}( A_i)}\right)^{\frac1{2p_i}}\|f_i\|_{p_i}
   =\Gamma_\rho\prod_{i=1}^m\|f_i\|_{p_i}.
 \end{align*}
 The proof for the other inequalities is identical and will be omitted.
\end{proof}

\section{Applications of Theorem \ref{THE prop}}\label{Sec6}


\subsection{Convolution inequalities}\label{sec6.2}

Recall the notations from $(A1)$ and $(A2)$. Assume that $m=2$, $n_{1}=n_{2}=n$, 
$N=2n$ and for any $\lambda \in [0,1]$, consider the following trivial 
decomposition of the identity in $\IR^n$
\[
 \lambda\,I_n + (1-\lambda)\,I_n = I_n.
\]
Set  $c_1=\lambda$, $c_2=1-\lambda$, $U_1=U_2=A=I_n$ and 
$W_1=\sqrt{1-\lambda}I_n$, $W_2=-\sqrt{\lambda} I_{n}$. Then 
a direct computation shows that \eqref{sec6.1:add:eq1} and 
\eqref{THEprop-Leh-2} hold and thus, Theorem \ref{THE prop} reads

\begin{proposition}\label{THEprop-m2} 
 Let $f_1,f_2$ be nonnegative measurable functions on $\mathbb R^{n}$ and 
 $\lambda\in [0,1]$. Define the function 
 \[
 F_\lambda(x,y)=f_1\biggl(x+\sqrt{\frac{1-\lambda}{\lambda}}y\biggr)f_2\,
 \biggl(x-\sqrt{\frac{\lambda}{1-\lambda}}y\biggr), \quad (x,y)\in
 \mathbb{R}^n\times\mathbb{R}^n
 \]
 and for $\rho>0,$ set
 \begin{align*}
  p_1=\frac{\rho}{(\rho-1)\lambda+1},\qquad 
  p_2=\frac{\rho}{(\rho-1)(1-\lambda)+1},
 \end{align*}
 and
 \begin{align*}
  \Im_{\rho}=
  \left( \frac{\lambda^{ \frac{1}{p_1}}(1-\lambda)^{\frac{1}{p_2}}\;\;
  p_1^{\frac{1}{p_1}} p_2^{\frac{1}{p_2}} }{ \rho^{\frac{2}{\rho}} } 
  \right)^{\frac{n}{2}}.
 \end{align*}
 $(i)$ If $\rho \geq 1,$ then
 \begin{equation} \label{sec6.2:eq1}
  \int_{\mathbb{R}^{n}} \left( \int_{\mathbb{R}^n} 
  F_\lambda^\rho(x,y) \,dy\right)^{\frac{1}{\rho}} dx 
  \,\geq\, \Im_\rho \, \|f_1\|_{p_1} \, \|f_2\|_{p_2} \,\geq\,
  \left( \int_{\mathbb{R}^n} \left( \int_{\mathbb{R}^n} 
  F_\lambda(x,y) \,dx \right)^{\rho} dy \right)^{\frac1\rho}.
 \end{equation}
 $(ii)$ If $0\leq \rho\leq 1,$ then
 \begin{equation} \label{sec6.2:eq2}
  \int_{\mathbb{R}^{n}} \left( \int_{\mathbb{R}^n} 
  F_\lambda^\rho(x,y) \,dy\right)^{\frac{1}{\rho}} dx 
  \,\leq\, \Im_\rho \, \|f_1\|_{p_1} \, \|f_2\|_{p_2} \,\leq\,
  \left( \int_{\mathbb{R}^n} \left( \int_{\mathbb{R}^n} 
  F_\lambda(x,y) \,dx \right)^{\rho} dy \right)^{\frac1\rho}.
 \end{equation}
\end{proposition}

\medskip

\noindent \textbf{Sharp Young and reverse Young inequalities}. 
By a change of variables, Proposition \ref{THEprop-m2} can also be read as

\begin{proposition} \label{THEprop-conv}
 Let $f_1,f_2$ be non-negative measurable functions on $\mathbb R^{n}$.
 For any $\lambda\in(0,1)$, let $p_1$, $p_2$ and $\Im_\rho$ be as in 
 Proposition \ref{THEprop-m2} and set 
 \[
 \Im_\rho'=\frac{\Im_\rho}{\bigl(\lambda(1-\lambda)\bigr)^{\frac{n}{2\rho}}}.
 \]
 $(i)$ If $\rho\geq 1$ then
 \begin{equation} \label{THEprop-conv-2} 
  \| f_1\ast f_2\|_{\rho} 
  \,\leq\, \Im_\rho'\,\|f_1\|_{p_1}\;\|f_2\|_{p_2} \,\leq\,
  \| (f_1^{\rho} \ast f_2^{\rho})^{\frac1\rho}\|_1.
 \end{equation}
 $(ii)$ If $0\leq \rho\leq 1,$ then
 \begin{equation} \label{THEprop-conv-3}
 \| f_1\ast f_2\|_{\rho} 
  \,\geq\, \Im_\rho'\,\|f_1\|_{p_1}\;\|f_2\|_{p_2} \,\geq\,
  \| (f_1^{\rho} \ast f_2^{\rho})^{\frac1\rho}\|_1.
 \end{equation}
\end{proposition}

This is indeed, a reformulation of the sharp Young and reverse Young inequalities. 
To see this, suppose that $p,q,r>0$ satisfy $p^{-1}+q^{-1}=1+r^{-1}$. Choose 
$\rho=r$ and $\lambda=r'/q'$ in Proposition \ref{THEprop-conv}, where $r',q'$ 
are conjugate exponents of $r,q$, respectively. Then $p_1= p$, $p_2= q$ and 
$\Im_\rho'=C^n$, where $C$ is defined in \eqref{sy} and \eqref{rsy}. 
If $p,q,r\geq 1,$ the left-hand side of \eqref{THEprop-conv-2} gives \eqref{sy}, 
while if $0<p,q,r<1,$ the left-hand side of \eqref{THEprop-conv-3} gives \eqref{rsy}.

\medskip

\noindent \textbf{Prekopa-Leindler inequality}. Letting $\rho\rightarrow\infty$ 
in Proposition \ref{THEprop-m2}, the right-hand side of \eqref{sec6.2:eq1} gives 
H\"{o}lder's inequality. As for the left-hand side, Pr\'ekopa-Leindler inequality 
\cite{Lei,Pre}, which is the functional form of the Brunn-Minkowski inequality, 
the cornerstone of the Brunn-Minkowski theory. For more information on this subject, 
we refer the reader to the book  \cite{Sch} and the survey paper of Gardner \cite{Gardner}.

\begin{theorem}[Pr\'ekopa-Leindler's inequality]\label{add:thm1} 
 Let $f,g,h$ be three nonnegative measurable functions in $\mathbb
 R^{n}$ and $\lambda \in [0,1]$ such that
 \begin{equation}
 \label{Pre-Lei-1} h(\lambda x + (1-\lambda) y) \geq f^{\lambda}(x)
 g^{1-\lambda}(y), \ \forall x,y \in \mathbb R^{n} .
 \end{equation}
 Then
 \begin{equation} \label{Pre-Lei-2} 
  \int_{\mathbb R^n} h(x) dx  \geq
  \left(\int_{\mathbb R^{n}} f(x) d x \right)^{\lambda}
  \left(\int_{\mathbb R^{n}} g (x) dx\right)^{1-\lambda} .
 \end{equation}
\end{theorem}

\begin{proof}
 Applying the left-hand side inequality of \eqref{sec6.2:eq1} to $f_1:=f^{\lambda}$ 
 and $f_2:=g^{1-\lambda}$ and sending $\rho$ to $\infty$, we get that
 \begin{align*}
  &\int_{\mathbb R^{n}} {\rm essup}_{y\in \mathbb R^{n}}
   f^{\lambda} \left(x+ \sqrt{\frac{1-\lambda}{\lambda}} y \right)
   g^{1-\lambda} \left(x- \sqrt{\frac{\lambda}{1-\lambda}} y \right)dx 
  &\geq  \left(\int f \right)^{\lambda} \left(\int g \right)^{1-\lambda},
 \end{align*}
 Notice that from \eqref{Pre-Lei-1},
 \begin{align*}
  &{\rm essup}_{y\in \mathbb R^{n}} 
   f^{\lambda} \left(x+\sqrt{\frac{1-\lambda}{\lambda}} y \right) 
    g^{1-\lambda} \left(x-\sqrt{\frac{\lambda}{1-\lambda}} y\right)  \\
  &\leq {\rm essup}_{y\in \mathbb R^n} 
        h\left( 
        \lambda\, \Bigl(x+\sqrt{\frac{1-\lambda}{\lambda}} y \Bigr)
        +(1-\lambda) \Bigl(x-\sqrt{\frac{\lambda}{1-\lambda}} y\Bigr) 
        \right)  = h(x).
 \end{align*}
 This gives \eqref{Pre-Lei-2}.
\end{proof}

\begin{remark}\rm
 The proof of Theorem \ref{add:thm1} actually gives the essential supremum
 strengthened version of the Pr\'ekopa-Leindler's inequality, proved by 
 Brascamp and Lieb in \cite{BL2}, which also avoids problems of measurability. 
 We refer to the Appendix of \cite{BL2} and Section 9 in \cite{Gardner} for 
 more details.
\end{remark}

\subsection{Brascamp-Lieb and Barthe inequalities}\label{sec6.3}

Theorem \ref{THE prop}, without the restriction $m=2$, leads to the Brascamp-Lieb 
and Barthe inequalities by letting again $\rho$ in \eqref{THEprop-Leh-4} tend to 
infinity. 

\begin{theorem} \label{Leh} 
 Assume that $(A1)$ and $(A2)$ in Section \ref{sec6} hold. Then
 \begin{itemize}
  \item[$(i)$]  {\rm(Brascamp-Lieb's inequality)}. 
                For any nonnegative measurable functions $f_{i}$ on 
                $\mathbb{R}^{n_i}$, $i\leq m$, we have that
                \begin{equation} \label{Leh-3}
                 \int_{\mathbb R^{n} }\prod_{i=1}^{m} f_{i}(U_{i}x)\,dx 
                 \,\leq \left(\frac{{\rm det}(A)}{\prod_{i=1}^m
                 {\rm det}(A_i)^{c_i}}\right)^{\frac12}  
                 \prod_{i=1}^m \|f_i \|_{\frac{1}{c_i}}.
                \end{equation}
                The equality holds if  
                \[
                f_{i}(x_i)=\exp(-c_i\langle A_i^{-1} x_i,x_i \rangle),
                \quad i\leq m
                \]                 
  \item[$(ii)$] {\rm(Barthe's inequality)}.
                For any nonnegative measurable functions $f_i$ on 
                $\mathbb{R}^{n_i}$, $i\leq m$ and $f$ on $\mathbb{R}^n$ 
                that satisfy
                \begin{equation} \label{Leh-4}
                 f\left( \sum_{i=1}^{m} c_{i} U_{i}^{\ast} x_{i}\right) 
                 \,\geq \prod_{i=1}^{m} f_{i}(x_{i}), \quad  x_i\in\IR^{n_i}
                \end{equation}
                 we have that
                \begin{equation} \label{Leh-5}
                 \prod_{i=1}^{m} \|f_{i} \|_{\frac{1}{c_{i}}} 
                 \leq \left( \frac{ {\rm det}(A)}{ 
                 \prod_{i=1}^{m} {\rm det} ( A_i)^{c_{i}}}\right)^{\frac{1}{2}}  
                 \int_{\mathbb R^{n}} f (x)  d x.
                \end{equation}
                The equality holds if 
                \[
                f_{i}(x_i)= \exp(-c_i\langle A_i x_i,x_i \rangle/2),\quad i\leq m,
                \]
                and
                \[
                f(x)= \exp(-\langle Ax, x\rangle/2).
                \]
 \end{itemize}
\end{theorem}

\begin{remark}\rm 
 Theorem \ref{Leh} actually retrieves the recent work of Lehec in \cite{Le} 
 and differs from the initial statements of the Brascamp-Lieb and Barthe inequalities. 
 However, the author in \cite{Le} provides an argument on how the initial statements
 can always be recovered from Theorem \ref{Leh}.
 \end{remark}

\begin{proof}[Proof of Theorem \ref{Leh}] 
 The statements that equalities can be realized by the given functions in both cases 
 can be easily verified by direct computations. We will omit this part of the argument. 
 Recall $W$ from \eqref{THEprop-Leh-2}. To show \eqref{Leh-3}, sending $\rho$ in the 
 right-hand side of \eqref{THEprop-Leh-4} to infinity, one gets
 \begin{align*}
 &\sup_{y\in \mathbb R^{N-n}}  \int_{\mathbb R^{n}} 
  \prod_{i=1}^{m} f_{i} \biggl(  {{U}}_{i}x +
  \frac{1}{\sqrt{c_{i}}} \sqrt{A_i}W_{i} y \biggr) dx\\
 &\leq \left( \frac{ {\rm det}(A)}{ \prod_{i=1}^{m} 
  {\rm det} ( A_i)^{c_{i}}}\right)^{\frac{1}{2}}  
  \prod_{i=1}^{m} \|f_{i} \|_{\frac{1}{c_{i}}}.
 \end{align*}
 Note that here we have used the condition \eqref{Leh-1} in the limit 
 $\lim_{\rho\rightarrow\infty}\Gamma_\rho$.
 The inequality \eqref{Leh-3} then follows by observing that
 \[ 
 \int_{\mathbb R^{n}} \prod_{i=1}^{m} f_{i} 
 \big(  {{U}}_{i}x \big) dx \leq  \sup_{y\in \mathbb R^{N-n}}  
 \int_{\mathbb R^{n}} \prod_{i=1}^{m} f_{i}\biggl({{U}}_{i}x +
 \frac{1}{\sqrt{c_{i}}}\sqrt{A_i}W_{i} y \biggr) dx.
 \]
 As for \eqref{Leh-5}, define $g_{i}:\mathbb R^{n_{i}} \rightarrow \mathbb R_{+}$ 
 by $g_{i}(x_i)= f_{i}( A_{i}^{-1} x_i)$. Note 
 $\|g_{i}\|_{p}={\rm det} (A_{i})^{1/p}\|f_{i}\|_{p}$. Applying the left-hand side 
 of \eqref{THEprop-Leh-4} for the functions $g_{i}$'s and then passing to the limit 
 $\rho\rightarrow\infty$ by using \eqref{Leh-1}, we obtain
 \begin{align*}
  &\int_{\mathbb R^{n}} \sup_{y\in \mathbb R^{N-n}}  
   \prod_{i=1}^{m} f_{i}\left( \frac{1}{ \sqrt{c_{i}}}A_{i}^{-\frac{1}{2}} 
   (\sqrt{c_{i}} A_{i}^{-\frac{1}{2}} U_{i} x + W_{i} y ) \right) d x\\
  &\geq{\rm det} (A)^{\frac{1}{2}}  \prod_{i=1}^{m}
   {\rm det}(A_{i})^{\frac{1}{2p_{i}}} \prod_{i=1}^{m} \|f_{i}\|_{p_{i}}
 \end{align*}
 and by the change of variable $x= Az$,
 \begin{align*}
  &\int_{\mathbb R^{n}} \sup_{y\in \mathbb R^{N-n}}  
   \prod_{i=1}^{m} f_{i}\left(\frac{1}{\sqrt{c_{i}}}A_{i}^{-\frac{1}{2}} 
   ( \overline{U}_{i} \sqrt{A} z + W_{i} y ) \right)d z\\
  &\geq    \frac{ \prod_{i=1}^{m}   
   {\rm det} (A_{i})^{\frac{1}{2p_{i}}} }{{\rm det} (A)^{\frac{1}{2}}}
   \prod_{i=1}^{m} \|f_{i}\|_{p_{i}},
 \end{align*}
 where ${\overline{U}}_{i} := \sqrt{c_{i}} A_{i}^{-1/2} U_{i} \sqrt{A} $. 
 From \eqref{Leh-4}, \eqref{Leh-5} will be valid if the following holds
 \begin{equation} \label{proof-Ba-4}
 \sup_{y\in \mathbb R^{N-n}}  \prod_{i=1}^{m} f_{i}
 \left( \frac{1}{ \sqrt{c_{i}}}A_{i}^{-\frac{1}{2}} 
 (\overline{U}_{i} \sqrt{A} z + W_{i} y ) \right) = 
 \sup_{(\xi_1,\ldots,\xi_m):\sum_{i=1}^mc_iU_i^*\xi_i=z} 
 \prod_{i=1}^{m} f_{i} (\xi_{i}).
 \end{equation}
 To show this identity, we first claim that for functions 
 $F_{i}:\mathbb R^{n_{i}}\rightarrow \mathbb R_{+}$ for $i\leq m$, we have
 \begin{equation} \label{proof-Ba-5}
 \sup_{y\in \mathbb R^{N-n}} \prod_{i=1}^{m} F_{i} 
 \left( V_{i} ( x,y)\right) = 
 \sup_{(a_1,\ldots,a_m):\sum_{i=1}^m\overline{U}_i^*a_i=x } 
 \prod_{i=1}^{m} F_{i} (a_{i}).
 \end{equation}
 Recalling \eqref{THEprop-Leh-2}, if we set $V_{i}=( {\overline{U}}_{i},W_{i})$ 
 for $i\leq m$, then the rows of $V_1,\ldots,V_m$ form an orthonormal basis of 
 $\mathbb R^{N}$. Suppose that 
 $a=(a_{1}, \ldots, a_{m})\in\mathbb{R}^{n_1}\times\cdots\times \mathbb{R}^{n_m}$ 
 satisfies $\sum_{i=1}^{m} \overline{U}_{i}^{\ast} a_{i}= x $. If we set 
 $y= \sum_{i=1}^{m} W_{i} a_{i}\in \mathbb R^{N-n}$, then $V_{i}( x,y)= a_{i}$. 
 This proves $\geq$ in \eqref{proof-Ba-5}. The proof of $\leq$ is similar and this 
 completes the proof of \eqref{proof-Ba-5}. Finally, applying \eqref{proof-Ba-5} to 
 \[
 F_i(x_i)=f_i\biggl(\frac{1}{\sqrt{c_i}}A_i^{-1}x_i\biggr)
 \]
 for $x_i\in\mathbb{R}^{n_i},$ we have that
 \begin{align*}
  &\sup_{y\in \mathbb R^{N-n}}  \prod_{i=1}^{m} f_{i}
   \left( \frac{1}{ \sqrt{c_{i}}}A_{i}^{-\frac{1}{2}} 
   (\overline{U}_{i} \sqrt{A} z + W_{i} y ) \right)\\
  & =\sup_{y\in \mathbb R^{N-n}}  \prod_{i=1}^{m} f_{i}
     \left(\frac1{\sqrt{c_i}}A_i^{-\frac12} V_i(\sqrt A z,y)\right)\\
  & =\sup_{(a_1,\ldots,a_m):\sum_{i=1}^m\overline{U}_i^*a_i=
     \sqrt{A}z} \prod_{i=1}^{m}f_{i}\left( \frac{1}
     {\sqrt{c_{i}}}A_{i}^{-\frac{1}{2}}a_{i} \right)\\
  & =\sup_{(\xi_1,\ldots,\xi_m):\sum_{i=1}^mc_iU_i^*\xi_i=z} 
     \prod_{i=1}^{m} f_{i} (\xi_{i}),
 \end{align*}
 where the last equality used change of variables $\xi_i=c_i^{-1/2}A_i^{-1/2}a_i$. 
 This gives \eqref{proof-Ba-4} and we are done.
\end{proof}

\subsection{An entropy inequality}\label{sec6.4}

Finally we derive entropy inequalities for probability density functions. 
Let $f$ be a positive measurable function in $\mathbb R^n$. We define the entropy 
of $f$ by
\begin{equation*}
 {\rm Ent}(f):=
 \int_{\mathbb R^{n}} f(x)\log{f}(x)\,dx - 
 \left( \int f(x)dx \right) \log{\int f(x)dx},
\end{equation*}
whenever this quantity makes sense. Note that if $g(p):=\|f\|_{{p}},$ then
\begin{equation} \label{ent-der} 
 g^{\prime}(1) = {\rm Ent}(f) . 
\end{equation}
Let $F$ and $\Gamma_{\rho}$ be defined as in Theorem \ref{THE prop} and set the 
functions $G_{1}, G_{2}, G_{3}$ on $\mathbb [0,\infty)$ by
\begin{align*}
 G_1(\rho)&= \int_{\mathbb R^{n}} \left(\int_{\IR^{N-n}} 
              F^\rho(x,y)\,dy\right)^{\frac{1}{\rho}}dx, \\
 G_{2}(\rho)&=\Gamma_\rho\prod_{i=1}^m\|f_{i}\|_{p_{i}},\\
 G_{3}(\rho)&=\left( \int_{\mathbb R^{N-n}}\left(\int_{\mathbb R^{n}}
              F(x,y)\,dx \right)^{\rho} dy \right)^{\frac{1}{\rho}}.
\end{align*}
Note that Fubini's theorem implies that $G_{1}(1)= G_{2}(1)= G_{3}(1)$ and Theorem 
\ref{THE prop} states that 
\[
 G_{1}(\rho) \leq G_{2}(\rho) \leq G_{3}(\rho), \quad {\rm if} \ \rho\leq 1,
\]
and
\[
  G_{1}(\rho) \geq G_{2}(\rho) \geq G_{3}(\rho), \quad {\rm if} \ \rho\geq 1.
\]
Putting all these together gives
\begin{equation}
\label{ent-der-1}
G_{3}^{\prime} (1) \leq G_{2}^{\prime} (1) \leq G_{1}^{\prime} (1), 
\end{equation}
which leads to the following entropy inequalities.

\begin{proposition} \label{prop-der}
 Assume that $(A1)$ and $(A2)$ hold and $W$ satisfies \eqref{THEprop-Leh-2}. 
 For any probability density $g_i$ on $\mathbb{R}^{n_i}$, $i\leq m$, set
 \begin{equation} \label{ent-G}
  G(x,y)= \prod_{i=1}^{m} g_{i} \left( \sqrt{c_{i}}U_{i} x + W_{i}y\right)
 \end{equation}
 Then 
 \begin{equation} \label{ent-0}
  D_1\,{\rm Ent}\Bigl(\int_{\IR^n} G(x,\cdot)\,dx\Bigr)  
  \,\leq\, \sum_{i=1}^{m} (1-c_{i})\,{\rm Ent}(g_{i}) + D_2 \,\leq\, 
  D_1\int_{\IR^n} {\rm Ent}\bigl(G(x,\cdot)\bigr)\,dx, 
 \end{equation}
 where 
 \begin{align*}
  D_1:=\left( \frac{ \prod_{i=1}^{m}  
  {\rm det}(A_{i})}{{\rm det}(A)} \right)^{\frac{1}{2}}
  \quad{\rm and}\quad\;      
  D_2:=\frac{1}{2}\sum_{i=1}^m(1-c_i)\log\det(A_i).
 \end{align*}
\end{proposition}

\begin{proof} The idea of the proof is to compute the derivatives of $G_1$, $G_2$ 
 and $G_3$ at $\rho=1$. One shall see that they lead to the three quantities in 
 \eqref{ent-0} and the inequalities are preserved through \eqref{ent-der-1}. 
 Note first that from \eqref{ent-der}, we have that 
 $G_1'(1)=\int_{\mathbb{R}^n}{\rm Ent}\bigl(G(x,\cdot)\bigr)\,dx$ 
 and $G_3'(1)={\rm Ent}\bigl(\int_{\IR^n} G(x,\cdot)\,dx\bigr)$. 
 As for $G_2'(1)$, recalling that $\Gamma_\rho$ from \eqref{THE prop} 
 and defining $\Omega_\rho= \prod_{i=1}^{m} \|f_{i}\|_{p_{i}}$, we get 
 directly by definition that
 \begin{align*}
  \Gamma_1&=\left(\det(A)\prod_{i=1}^m 
  \frac{c_i^{n_i}}{\det(A_i)}\right)^{\frac{1}{2}},
  \,\,\Omega_1=\prod_{i=1}^m\|f_i\|_1,
 \end{align*}
 and a quite tedious computation yields
 \begin{align*}
  \left.\frac{d\Gamma_\rho}{d\rho}\right|_{\rho=1}
   &={\det(A)}^{\frac{1}{2}}\left(\prod_{i=1}^m
     \frac{c_i^{n_i/2}}{\sqrt{\det(A_i)}}\right)
     \left(\sum_{i=1}^m\frac{(1-c_i)n_i}{2}
     \left[\frac{\log\det(A_i)}{n_i}-\log c_i\right]\right),\\
     \left.\frac{d\Omega_\rho}{d\rho}\right|_{\rho=1}
   &=\left( \prod_{i=1}^{m} \|f_{i}\|_{1}\right) 
     \left( \sum_{i=1}^{m} (1-c_{i}) 
     \frac{{\rm Ent} (f_{i})}{ \|f_{i}\|_{1}} \right).
 \end{align*}
 Combining these all together gives 
 \begin{align*}
  G_2'(1)&=\left( \prod_{i=1}^{m} \|f_{i}\|_{1}\right) 
           \left( {\rm det}(A) \prod_{i=1}^{m} 
           \frac{c_{i}^{n_{i}}}{  {\rm det}(A_{i})} 
           \right)^{\frac{1}{2}} \\
         &\times \left( \sum_{i=1}^{m} \frac{n_{i}(1-c_{i})}{2} 
         \left( \frac{ 2{\rm Ent}(f_{i})}{ n_{i} \|f_{i}\|_{1}}+
         \frac{\log{{\rm det}(A_{i})}}{n_{i}}-\log{c_{i}}\right) \right).
 \end{align*}
 Set $f_{i}(x):= g_{i}(\sqrt{c_{i}}x)$. Observe that 
 \begin{align*}
  c_{i}^{\frac{n_{i}}{2}}\int f_i&=\int g_i=1
  \quad{\rm and}\quad
  c_{i}^{\frac{n_{i}}{2}} {\rm Ent} (f_{i}) 
  = {\rm Ent}(g_{i}) + \frac{n_{i}}{2} \log{c_{i}}.
 \end{align*}
 So
 \begin{align*}
  \frac{2{\rm Ent}(f_i)}{ n_i \|f_{i}\|_1}
  = \log{c_i}+\frac{2}{n_i} {\rm Ent}(g_i)
 \end{align*}
 and thus,
 \begin{align*}
  G_{2}^{\prime} (1)
   &= \left( \frac{{\rm det}(A)}{ \prod_{i=1}^{m}  
      {\rm det}(A_{i})} \right)^{\frac{1}{2}}
      \left(\sum_{i=1}^m(1-c_i){\rm Ent}(g_i)+
      \frac{1}{2}\sum_{i=1}^m(1-c_i)\log \det(A_i)\right)\\
   &= D_1^{-1}\sum_{i=1}^m(1-c_i)\mbox{Ent}(g_i)+D_1^{-1}D_2.
 \end{align*}
 Using our computations for $G_1'(1),$ $G_2'(1)$ and $G_3'(1),$ \eqref{ent-der-1} completes our proof.
\end{proof}

\begin{Acknowledgements}
 It is our pleasure to thank Dario Cordero-Erausquin, Petros Valettas and Joel Zinn for
 useful discussions, Dmitry Panchenko for many valuable remarks and his encouragement
 to put forth this work and Assaf Naor for bringing the results of \cite{MOS} to
 our attention. Special thanks are due to Franck Barthe and Pawel Wolff for informing 
 us that a similar result as Theorem \ref{THE prop} have also been obtained in their 
 recent research \cite{BW}. 
 
 \smallskip
 
 The second author is supported by the action Supporting Postdoctoral Researchers 
 of the operational program Education and Lifelong Learning (Action's Beneficiary: 
 General Secretariat for Research and Technology) and is co-financed by the European 
 Social Fund (ESF) and the Greek State. 
 
 \smallskip
 
 The last author is supported by the A. Sloan foundation, BSF grant 2010288 and the 
 US NSF grant CAREER-1151711.
\end{Acknowledgements}



\end{document}